\documentclass[final,1p,times]{elsarticle}
\usepackage{amsmath,amssymb,amsthm}
\usepackage{enumerate}
\usepackage{geometry}
\usepackage{graphicx}
\usepackage{svg}
\usepackage{subfigure}
\usepackage{epsfig}
\usepackage{graphicx}
\usepackage{colordvi}
\usepackage{graphics}
\usepackage{float}
\usepackage{ mathrsfs }
\usepackage{dutchcal}
\usepackage{tikz}

\numberwithin{equation}{section}
\newtheorem{theorem}[equation]{Theorem}

\newtheorem{lemma}[equation]{Lemma}
\newtheorem{remark}[equation]{Remark}

\newtheorem{algorithm}{Algorithm}[section]

\newcommand{\dt}{\Delta t}
\newcommand{\unone}{u^{n+1}}

\newcommand{\utn}{\tilde{u}^n}
\newcommand{\enone}{e^{n+1}}

\newcommand{\etnone}{\tilde{e}^{n+1}}
\newcommand{\etn}{\tilde{e}^n}

\newcommand{\eps}{{\varepsilon}}

\def\div{\operatorname{div}}
\journal{arXiv}

\begin{document}

\begin{frontmatter}

%% Title, authors and addresses

%% use the tnoteref command within \title for footnotes;
%% use the tnotetext command for theassociated footnote;
%% use the fnref command within \author or \address for footnotes;
%% use the fntext command for theassociated footnote;
%% use the corref command within \author for corresponding author footnotes;
%% use the cortext command for theassociated footnote;
%% use the ead command for the email address,
%% and the form \ead[url] for the home page:
%% \title{Title\tnoteref{label1}}
%% \tnotetext[label1]{}
%% \author{Name\corref{cor1}\fnref{label2}}
%% \ead{email address}
%% \ead[url]{home page}
%% \fntext[label2]{}
%% \cortext[cor1]{}
%% \affiliation{organization={},
%%             addressline={},
%%             city={},
%%             postcode={},
%%             state={},
%%             country={}}
%% \fntext[label3]{}

%% use optional labels to link authors explicitly to addresses:
%% \author[label1,label2]{}
%% \affiliation[label1]{organization={},
%%             addressline={},
%%             city={},
%%             postcode={},
%%             state={},
%%             country={}}
%%
%% \affiliation[label2]{organization={},
%%             addressline={},
%%             city={},
%%             postcode={},
%%             state={},
%%             country={}}

\title{Removing splitting/modeling error in projection/penalty methods for Navier-Stokes simulations with continuous data assimilation}
\author[label1]{Elizabeth Hawkins}
\ead{evhawki@clemson.edu}
\author[label1]{Leo G. Rebholz\corref{cor1}\fnref{lab1}}
\ead{rebholz@clemson.edu}
\author[label1]{Duygu Vargun\fnref{lab1}}
\ead{dvargun@clemson.edu}
\address[label1]{Department of Mathematical Sciences, Clemson University, Clemson, SC 29634, USA. }

\cortext[cor1]{Corresponding  author.}

\begin{abstract}
We study continuous data assimilation (CDA) applied to projection and penalty methods for the Navier-Stokes (NS) equations.  Penalty and projection methods are more efficient than consistent NS discretizations, however are less accurate due to modeling error (penalty) and splitting error (projection).  We show analytically and numerically that with measurement data and properly chosen parameters, 
CDA can effectively remove these splitting and modeling errors and provide long time optimally accurate solutions.
%Additionally, we derive an altered projection method that has a weighted $H^1$ projection into the div free space instead of the usual $L^2$ projection, in order to improve the effect of CDA in the projection step. 
%Numerically tests show the CDA-altered projection method performs significantly better than CDA-projection method on the flow past a block numerical test.
\end{abstract}


\end{frontmatter}

\section{Introduction}\label{introsection}
Data assimilation has become a critical tool to improve simulations of many physical phenomena, from climate science to weather prediction to environmental forecasting and beyond  \cite{Daley_1993_atmospheric_book,Kalnay_2003_DA_book,Law_Stuart_Zygalakis_2015_book}.  While there are many types of data assimilation, one with perhaps the strongest mathematical foundation for use with PDEs that predict physical behavior is called Continuous Data Assimilation (CDA).  CDA was developed by Azouani, Olson, and Titi in 2014 \cite{AOT14}, and has since been successfully used on a wide variety of problems including Navier-Stokes equations \cite{AOT14}, Benard convection \cite{FJT15}, planetary geostrophic models \cite{FLT16}, turbulence \cite{DMB20,FLT19}, Cahn-Hilliard \cite{DR22} and many others.  Many improvements to CDA itself have also been made, through techniques for parameter recovery \cite{CHL20}, sensitivity analysis with CDA \cite{CL21}, numerical analysis \cite{IMT20,LRZ19,RZ21,DR22,GNT18}, and efficient nudging methods \cite{RZ21}, to name just a few.

CDA is typically applied in the following manner.  Suppose the following PDE is the correct model for a particular physical phenomenon with solution $u(x,t)$:
\begin{align*}
u_t + F(u) &= f, \\
u(x,t)|_{\partial\Omega} & =0,\\
u(x,0)&=u_0(x).
\end{align*}
Suppose further that part of the true solution is known from measurements or observables, so that $I_H(u)$ is known at all times, with $I_H$ representing an appropriate interpolant with max point spacing $H$.  Then the CDA model takes the form
\begin{align*}
v_t + F(v) + \mu I_H(v-u) & = f,\\
v(x,t)|_{\partial\Omega} &=0,\\
v(x,0)&=v_0(x),
\end{align*}
where $\mu>0$ is a user selected nudging parameter.  For many such systems, given enough measurement values it can be proven that the solution $v$ is long time accurate regardless of the accuracy of the initial condition $v_0$ (often CDA analyses assume $v_0=0\ne u_0$).  In numerical analyses, accuracy results of CDA enhanced discretizations can often avoid error growth in time since application of the Gronwall inequality can be avoided, leading to long time optimal accuracy results \cite{GNT18,RZ21,GN20}.

The purpose of this paper is to study CDA together with two commonly used discretizations of the Navier-Stokes equations (NSE), the projection method and the penalty method.
The projection method is a classical splitting method for the NSE developed independently by Chorin and Temam \cite{T69,C68}, and is based on a Hodge decomposition.  The penalty method removes the divergence constraint but replaces it with a divergence penalty in the momentum equation.  Both of these methods are more efficient than consistent discretizations, however they are not as accurate: projection methods have splitting error that reduces accuracy below optimal, and penalty methods have a consistency error on the order of the penalty coefficient.  We will show through analysis and numerical tests that CDA removes the splitting error in projection method and consistency error in penalty method.

To begin our introductory explanation, we start with the NSE system, which is given by
\begin{align}
		w_t+w\cdot \nabla w +\nabla q-\nu \Delta w &= f, \label{nse1} \\
		\nabla\cdot w&=0, \label{nse2} \\
		w|_{\partial\Omega}&=0, \label{nse3} \\
		w(0)&=w_0, \label{nse4}
\end{align} 
where $f$ represents external forcing, $\nu$ the kinematic viscosity, and with $w$ and $q$ representing the unknown velocity and pressure.  A consistent linearized backward Euler  temporal discretization takes the form
\begin{equation}\label{BElinNSE}
	\begin{aligned}
\frac{u^{n+1} - u^n}{\Delta t} + u^{n}\cdot \nabla u^{n+1} + \nabla p^{n+1} - \nu\Delta u^{n+1} & = f^{n+1},\\
\nabla \cdot u^{n+1} &=0, \\
u^{n+1}|_{\partial\Omega}&=0.
	\end{aligned}
\end{equation}
For simplicity, we consider the linearized backward Euler time stepping for our analysis, but we note that the same ideas can be applied to the analogous BDF2-type methods as well (e.g. those from \cite{GMS06}), although with additional technical details.  Our numerical tests use both time backward Euler and BDF2.

The linear systems associated with coupled discretizations such as those arising from \eqref{BElinNSE}, which are often called nonsymmetric saddle point systems, can be very difficult to solve.  While significant progress has been made in recent years \cite{benzi,CLLRW13,FMSW21,ESW14}, solving these systems when $\nu$ is small can be slow and sometimes not completely robust.  Projection and penalty methods both avoid the need to solve such linear systems, as we see below, and thus with these methods it is typically much easier to `get numbers'.

The linearized backward Euler projection method is formulated as the following two step solve process: \\

Proj Step 1: Find $u^{n+1}$:
\begin{align*}
\frac{u^{n+1} - \tilde u^n}{\Delta t} + \tilde u^{n}\cdot\nabla u^{n+1} - \nu\Delta u^{n+1} & = f^{n+1},\\
u^{n+1}|_{\partial\Omega}&=0.
\end{align*}

Proj Step 2: Project $u^{n+1}$ into the divergence-free space
\begin{align*}
\frac{\tilde u^{n+1} - u^{n+1}}{\Delta t} + \nabla p^{n+1} &=0, \\
\nabla \cdot \tilde u^{n+1} &=0,\\
\tilde u^{n+1} \cdot n|_{\partial\Omega}&=0.
\end{align*}

The projection method is much more efficient and robust than solving the saddle point system above.  Proj Step 1 is a convection-diffusion solve, and while not simple when $\nu$ is small it is still well studied.  Proj Step 2 is the same at each time step and symmetric, but also can be formulated as a pressure Poisson problem.  Hence solving the linear systems is a much simpler process with the projection method.  However, there are downsides to projection methods, including reduced accuracy and solutions that are not completely physical (either not divergence-free, or do not satisfy the boundary conditions).  Since their development in the late 1960s there have been many improvements to projection methods \cite{P97,GMS06,GMS12,LNRW17,BCM01}, but still there is a trade-off of accuracy vs. efficiency.   This lack of accuracy is evidenced in numerous ways.  First, as mentioned in \cite{S92}, while the coupled backward Euler method \eqref{BElinNSE} has $O(\Delta t)$ velocity accuracy  in the $L^2(0,T;H^1)$ natural energy norm, the projection method above cannot attain first order accuracy in this norm.  While it can achieve first order temporal accuracy in other norms, additional restrictions on the domain (e.g. $\Omega$ has the $H^2$ elliptic regularity property) are required that are not required for first order accuracy of the coupled scheme.  %Additionally, the pressure accuracy in projection methods is also suboptimal.

In addition to projection methods, we also consider penalty methods in this paper.  The linearized backward Euler penalty method takes the following form: 

\begin{align*}
\frac{u^{n+1} - u^n}{\Delta t} 
+  \tilde B(u^n, u^{n+1})
%+ u^{n}\cdot\nabla u^{n+1} 
- \nu\Delta u^{n+1} + \nabla p^{n+1} & = f^{n+1},\\
\nabla \cdot u^{n+1} + \eps p^{n+1} &=0,\\
u^{n+1}|_{\partial\Omega}&=0,
\end{align*}
where $\tilde B(u,v)=(u\cdot \nabla)v+\frac{1}{2}(\div u)v$ is the modified bilinear form introduced by Temam \cite{T68} to guarantee the stability of such systems.

By solving for $p^{n+1}$ in the conservation of mass equation and inserting it into the momentum equation, we get a system in terms of velocity only:
\begin{align*}
\frac{u^{n+1} - u^n}{\Delta t} + \tilde B(u^n, u^{n+1}) - \nu\Delta u^{n+1} -\eps^{-1} \nabla \div u^{n+1} & = f^{n+1},\\
u^{n+1}|_{\partial\Omega}&=0.
\end{align*}
Hence this system is also more efficient than the consistent discretization of \eqref{BElinNSE}, but carries a $O(\eps)$ consistency error \cite{S95}.  Since the matrix arising from the `grad-div' term
$ -\eps^{-1} \nabla (\nabla \cdot u^{n+1}) $ is singular, numerical issues arise if $\eps$ is taken too small and often this consistency/modeling error can be a dominant error source \cite{OR11}. For $\eps$ not too small, however, linear system solves are quite efficient, even with direct solvers \cite{OR11}.

The purpose of this paper is to improve both projection and penalty methods by incorporating CDA into their respective schemes. The CDA enhanced Proj scheme uses nudging in Proj Step 1, and is given by

CDA Proj Step 1: Find $u^{n+1}$:
\begin{align*}
\frac{u^{n+1} - \tilde u^n}{\Delta t} + \tilde u^{n}\cdot\nabla u^{n+1} - \nu\Delta u^{n+1}  + \mu  I_H(u^{n+1} - w^{n+1}) & = f^{n+1},\\
u^{n+1}|_{\partial\Omega}&=0.
\end{align*}

CDA Proj Step 2: Project $u^{n+1}$ into the divergence-free space
\begin{align*}
\frac{\tilde u^{n+1} - u^{n+1}}{\Delta t} + \nabla p^{n+1}  &=0, \\
\nabla \cdot \tilde u^{n+1} &=0,\\
\tilde u^{n+1} \cdot n|_{\partial\Omega}&=0.
\end{align*}
Nudging could also be applied to velocity projection in Proj Step 2, it does not make any significant change in analysis or numerical results.
%In our computations, we see little difference when taking $\mu_2>0$, and so for simplicity we restrict our analysis to the case of $\mu_2=0$ but still show results of $\mu_2>0$ in our numerical tests.  

The CDA enhanced penalty method, in velocity-only form, can be written as 
\begin{equation}\label{penalyCDA}
	\begin{aligned}
	\frac{u^{n+1} - u^n}{\Delta t} + \tilde B(u^n, u^{n+1})
%+  u^{n}\cdot\nabla u^{n+1} + \frac{1}{2}(\div u^{n})  u^{n+1}
- \nu\Delta u^{n+1} -\eps^{-1}\nabla\div u^{n+1}  
	 +\mu I_H(u^{n+1}-w^{n+1}) & = f^{n+1},\\
		u^{n+1}|_{\partial\Omega}&=0.
	\end{aligned}
\end{equation}

We prove that under certain parameter choices found in our analysis, with CDA both projection and penalty methods recover optimal accuracy in the $L^2(0,T;H^1)$ energy norm and yield long time accuracy in $L^2(\Omega)$. Our numerical tests illustrate these results for both the first order schemes we analyze, and for their BDF2 analogues with similar improvement from CDA.  Interestingly, our numerical tests show that CDA-Penalty appears to give better numerical results than CDA-Projection, and moreover CDA-penalty can be very accurate even with 
$\eps=1$.

This paper is organized as follows. In section \ref{prelimsection}, we introduce the necessary notation and preliminary results required in the following sections. In section \ref{errorsection}, we establish stability and convergence analysis of the CDA-Projection method globally in time in $L^2$ and prove the rate of convergence of our scheme is $O(\dt)$ in $L^2(0,T;H^1)$. In section \ref{penaltyDAerrorsection}, we study the convergence of the CDA-Penalty scheme. Lastly, section \ref{experimentsection} contains two numerical tests that illustrate the optimal convergence rates and efficiency of the CDA-Projection and CDA-Penalty methods, respectively, on the benchmark problem of the channel flow past a cylinder.

\section{Notation and Preliminaries}\label{prelimsection}
We consider $\Omega\subset \mathbb{R}^d$, $d=2,3$, to be open bounded Lipschitz domain. The $L^2(\Omega)$ norm and inner product will be denoted by $\|\cdot\|$ and $(\cdot,\cdot)$ respectively, while all other norms will be labeled with subscripts. 
Additionally, $\langle\cdot,\cdot\rangle$ is used to denote the duality pairing between $H^{-s}$ and $H^s_0(\Omega)$ for all $s>0$.

We denote the natural function spaces for velocity and pressure, respectively, by
\begin{align*}
X &\coloneqq H^1_0(\Omega)^d,\\
Q &\coloneqq L^2_0(\Omega),
\end{align*}
which satisfy the inf-sup stability condition given by
\begin{align*}
\inf_{q\in Q} \sup_{v\in X} \frac{(q,\nabla\cdot v)}{\|q\| \|\nabla v\|} \geq \beta >0.
\end{align*}
The dual norm of X will be denoted by $\|\cdot\|_{-1}$. In addition to the spaces $X$ and $Q$, we define
\begin{align*}
Y&=\{ u\in(L^2(\Omega))^d: \nabla\cdot u=0,\ u\cdot n|_{\partial\Omega}=0 \},\\
V&=\{ u\in X: \nabla\cdot u=0 \},
\end{align*}
and $P_Y$ is the orthogonal projector in $(L^2(\Omega))^d$ onto $Y$.  The tilde notation will be used to denote this operator, e.g.
$$\tilde u^{n+1}= P_Y u^{n+1}.$$
The Stokes operator is defined by
\begin{align*}
A u=-P_Y \Delta u,\ \forall u\in D(A)=V \cap (H^2(\Omega))^d,
\end{align*}
which is an unbounded positive self-adjoint closed operator in $Y$ with domain $D(A)$, and its inverse $A^{-1}$ is compact in $Y$.

Given $u\in Y$, by definition of $A$, $v=A^{-1} u$ is the solution of the following Stokes equations:
\begin{equation}\label{stokes}
\begin{aligned}
-\Delta v + \nabla p &= u,\\
\nabla\cdot v&=0,\\
v|_{\partial\Omega}&=0.
\end{aligned}
\end{equation}
In \cite{Shen92}, the regularity results for \eqref{stokes} give
\begin{align*}
\|A^{-1} u\|_{H^s} = \|v\|_{H^s} \leq c_1 \|u\|_{H^{s-2}},\ \text{for}\ s=1,2;\ \text{and}\\
(A^{-1}u,u)=(v,u)=-(\Delta v,v)+(\nabla p, v)=\|v\|^2\leq c_1^2 \|u\|^2_{-1}.
\end{align*}
Additionally, since $u=Av$, by the inf-sup condition we obtain
\begin{align*}
\|u\|_{-1}\leq c \sup_{w\in X} \frac{\langle u,w\rangle}{\|w\|} = c \sup_{w\in X} \frac{\langle Av ,w\rangle}{\|w\|} \leq c\|v\|,
\end{align*}
which implies that $(A^{-1}u,u)^{1/2}$ can be used as an equivalent norm of $H^{-1}$ for all $u\in Y.$

We now to introduce some operators which will be used in our analysis.  For $u,v,w\in X,$
	\begin{align*}
		\tilde B(u,v)& =(u\cdot\nabla)v+\frac{1}{2}(\div u)v,\\
		\tilde b(u,v,w)& =(\tilde B(u,v),w).
	\end{align*} 
	Equivalently, this last term can be written as
	\begin{align*}
		\tilde b(u,v,w)=\frac{1}{2}\left(  
		((u\cdot\nabla)v,w)-((u\cdot\nabla)w,v)
		\right),\ \forall u,v,w\in X,
	\end{align*} 
	which is the skew symmetric form of the nonlinear term. Hence,
	\begin{align}\label{tildebilinear}
		\tilde b(u,v,v)=0,\ \forall u,v\in X.
	\end{align}

The following lemma is proven in \cite{LRZ19} and used to obtain the long time accuracy result.
\begin{lemma}\label{series}
	For constant $\alpha >1$ and $B>0$ if a sequence of real numbers $\{x_n\}_{n=0}^\infty$ satisfies
	$$\alpha x_{n+1} \leq x_n +B$$ 
	then
	$$ x_{n+1} \leq x_0 \frac{1}{\alpha^{n+1}} + \frac{B}{\alpha-1}$$
\end{lemma}

\subsection{Discretization preliminaries}
A function space for measurement data interpolation is also needed. Hence we require a regular conforming mesh $\tau_H$ and define  $X_H = P_r(\tau_H)^2$ for some polynomial degree $r$.  We require that the coarse mesh interpolation operator $I_H$ used for data assimilation satisfies the following bounds:  for any $\phi \in X$,
\begin{align}
	\| I_H  (\phi) - \phi \| &\leq C_I H \| \nabla \phi\|, \label{interp1}
	\\ \| I_H(\phi) \| &\leq  C_{I} \|\phi \|. \label{interp2}
\end{align}

%The operator $A_{\eps}$ is positive self-adjoint from $H^2(\Omega)\cap H^1_0(\Omega)$ onto $L^2(\Omega)$ and that the powers $A^{\alpha}$ of $A(\alpha\in \mathbb{R})$ are well-defined.
%
%\begin{lemma}
%There exist a constant $C>0$ such that for $\eps$ sufficiently small, we have
%\begin{align}
%\|\Delta u\|&\leq C\|A_{\eps} u\|,\ \forall u\in H^{2}\cap H^1_0(\Omega),\label{penaltyoperator1}\\
%\|\nabla u\|&\leq C\|A_{\eps}^{1/2} u\|,\ \forall u\in H^1_0(\Omega),\label{penaltyoperator2}\\
%\|A_{\eps}^{-1} u\|&\leq C\|u\|_{-2},\ \forall u\in H^{-2}(\Omega),\label{penaltyoperator3}
%\end{align}
%\end{lemma}

%\begin{align*}
%&B(u,v)=(u\cdot\nabla)v,\ \tilde B(u,v)=(u\cdot\nabla)v+\frac{1}{2}(\div u)v\\
%&b(u,v,w)=(B(u,v),w),\ \tilde b(u,v,w)=(\tilde B(u,v),w)
%\end{align*}

\section{CDA-Projection method error analysis}\label{errorsection}

We now consider the error resulting from the CDA-Projection method, which we write as the following semi-discrete algorithm.  While this is only semi-discrete, no additional difficulties would arise from a finite element spatial discretization (other than accuracy being limited {\color{black}by} the spatial approximation accuracy) and thus we suppress the spatial discretization.

\begin{algorithm} \label{daproj}
Let $w$ be the solution of \eqref{nse1}-\eqref{nse2} for a given divergence-free $w^0\in X$ and forcing $f\in L^{\infty}(0,\infty;H^{-1}(\Omega))$.  The for $u^0=w^0$ and given nudging parameter $\mu \ge 0$, find $\{ u^{n},\ \tilde u^n \}$ for n=1,2,3,... via the time stepping algorithm:

CDA Proj Step 1: Find $u^{n+1}\in X$ satisfying
\begin{equation}
\frac{1}{\dt} \left( \unone - \utn,v \right)+ (\utn\cdot\nabla\unone,v) + \nu (\nabla \unone,\nabla v) + \mu (I_H(u^{n+1} - w^{n+1}),v) = (f^{n+1},v) \ \forall v\in X. \label{pdastep1} \\
\end{equation}			

CDA Proj Step 2: Find $\tilde u^{n+1}\in Y$ and $p^{n+1}\in Q$ satisfying
\begin{align}
\frac{1}{\dt} \left(\tilde u^{n+1} - u^{n+1},v \right)- (p^{n+1},\nabla \cdot v)  & = 0  \ \forall v\in Y, \label{pdastep2a} \\
(\nabla \cdot \tilde u^{n+1},q) & =0. \ \forall q\in Q. \label{step2b}
\end{align}		
\end{algorithm}
\begin{remark}
Although the CDA-Projection method algorithm applies nudging to Step 1 only, it could also be applied to Step 2.  However, the resulting analysis requires more effort but without any improvement in the result, and moreover numerical tests (omitted herein) showed no significant improvement over nudging with Step 1 only.  
%Typically, CDA is applied to equations with a Laplacian term, and this term is often critical for CDA to yield improvement in accuracy.  
%In section ???, we consider an altered projection method where a Laplacian term is present in the projection step, and the CDA-Altered-Projection method performs better when nudging both Steps 1 and 2.
\end{remark}

We first prove that Algorithm \ref{daproj} is long time stable, without any restriction on the time step size $\dt$.
\begin{lemma}\label{lemma:wellposedness}
	Let $f\in L^{\infty}(0,\infty;L^2)$ and  $w\in L^{\infty}(0,\infty;L^2)$. Then, for any $\dt>0$ and any integer $n>0$, the velocity solution to Algorithm \ref{daproj} satisfies
\begin{align}\label{stability}
	\|u^{n}\|^2 
	\leq
	\| u^{0}\|^2  \frac{1}{\left(1+\dt(\mu+\lambda C_P^{-2})\right)^{n}  } 
	+\frac{\nu^{-1}}{\mu+\lambda C_P^{-2}} \|f\|_{L^{\infty}(0,\infty;L^2)}^2  
	+\frac{2\mu  }{\mu+\lambda C_P^{-2}} \|w\|_{L^{\infty}(0,\infty;L^2)}^2,
\end{align}	
with $\mu H^2 < \frac{\nu}{2 C_I^2}$.
\end{lemma}
\begin{proof}
Choose $v=\tilde u^{n+1}$ and $q=p^{n+1}$ in \eqref{pdastep2a}-\eqref{step2b}, which vanishes the pressure term, and gives
\begin{align*}
\|\tilde u^{n+1} \|^2 = \left( u^{n+1}, \tilde u^{n+1}\right).
\end{align*}		
Then by the Cauchy-Schwarz inequality, we obtain 
\begin{align}\label{boundontilde}
	\|\tilde u^{n+1} \| \leq \|\tilde u^{n+1}\|.
\end{align}
Next, choose $v=u^{n+1}$ in \eqref{pdastep1} which vanishes the nonlinear term, and provides us with
\begin{align*}
	\frac{1}{2\dt} \left(\|u^{n+1}\|^2 - \|\tilde u^{n}\|^2 +\|u^{n+1} - \tilde u^{n}\|^2 \right)+ \nu \|\nabla  u^{n+1}\|^2 + \mu (I_H(u^{n+1} - w^{n+1}), u^{n+1}) = (f^{n+1}, u^{n+1}). 
\end{align*}	
Next, we add and subtract $u^{n+1}$ in the first component of the nudging term and multiply both sides by $2\dt$, which yields 
\begin{align*}
	\|u^{n+1}\|^2 
	+ 2\dt \nu &\|\nabla  u^{n+1}\|^2 
	+ 2\dt\mu\|u^{n+1}\|^2
	\\
	&=
	\|\tilde u^{n}\|^2
	+2\dt (f^{n+1}, u^{n+1}) 
	+2\dt \mu (I_H( w^{n+1}), u^{n+1}) 
	-2\dt \mu (I_H(u^{n+1} ) - u^{n+1}, u^{n+1}) ,
\end{align*}	
after dropping the positive term $\|u^{n+1} - \tilde u^{n}\|^2$ on the left hand side.

The first term on the right hand side is bounded using the dual norm of $X$ and Young’s inequality, which yields
 \begin{align*}
 	2\dt (f^{n+1}, u^{n+1}) \leq \dt \nu^{-1} \|f^{n+1}\|^2 + \dt \nu \|\nabla u^{n+1}\|^2.
 \end{align*}	
Then, for the interpolation terms, we use Cauchy-Schwarz and Young's inequalities to obtain
\begin{align*}
2\dt \mu (I_H( w^{n+1}), u^{n+1})\leq 2\dt\mu \| w^{n+1}\|^2 + \dt\frac{\mu}{2} \|u^{n+1}\|^2,
\end{align*}
thanks to the interpolation property \eqref{interp1}, and 
 \begin{align*}
|-2\dt \mu (I_H(u^{n+1} ) - u^{n+1}, u^{n+1})|\leq 2\dt \mu C_I^2 H^2 \|\nabla u^{n+1}\|^2 + \dt \frac{\mu}{2} \|u^{n+1}\|^2,
 \end{align*}
thanks to the interpolation property \eqref{interp2}.

Combining the above estimates produces the bound
\begin{align*}
	\|u^{n+1}\|^2 
	+ \dt \left(\nu -2 \mu C_I^2 H^2\right)\|\nabla  u^{n+1}\|^2 
	+ \dt \mu \|u^{n+1}\|^2
	\leq
	\|\tilde u^{n}\|^2
	+\dt \nu^{-1} \|f^{n+1}\|^2  
	+2\dt\mu \| w^{n+1}\|^2. 
\end{align*}	
Assuming $\lambda=\nu -2 \mu C_I^2 H^2 >0$ and applying the Poincar\'e inequality on the left hand side gives
\begin{align*}
	\left(1+\dt(\mu+\lambda C_P^{-2})  \right)\|u^{n+1}\|^2 
	\leq
	\| u^{n}\|^2
	+\dt \nu^{-1} \|f^{n+1}\|^2  
	+2\dt\mu \| w^{n+1}\|^2,
\end{align*}	
thanks to the \eqref{boundontilde}. 
Next, we apply Lemma \ref{series} which reveals \eqref{stability} with regularity assumptions on $f$ and true solution $w$.
\end{proof}

\begin{lemma}
Under the same assumptions as the previous lemma, Algorithm \ref{daproj} is well-posed.
\end{lemma}
\begin{proof}
At each time step, Algorithm \ref{daproj} is a type of linear Oseen problem with an additional nudging term.  With the regularity assumptions and with the long time $L^2$ stability of $u^n$ established, analysis from the proof for the nudging term can be combined with standard theory for Oseen equations to achieve well-posedness of each time step and thus also the entire algorithm.
\end{proof}

We now prove that CDA can remove the splitting error of the projection method.  More specifically, with properly chosen parameters, the solution to Algorithm \ref{daproj} is long-time first order accurate in the velocity, and finite time first order accurate in the $L^2(0,T;H^1)$ norm.  We found no improvement in accuracy for the CDA Proj pressure, since the CDA term will change the Hodge decomposition in a way that $p$ will still represent a Lagrange multiplier corresponding to the divergence constraint in the projection step, but can no longer be interpreted as the pressure. Instead, the pressure can be recovered by post-processing, \cite{Shen92, GMS06}.

\begin{theorem}\label{thm1}
Suppose $(w,q)$ is the solution to the NSE with $w\in L^{\infty}(0,\infty;H^3)$, $w_t,w_{tt}\in L^{\infty}(0,\infty;H^1)$, $q\in L^{\infty}(0,\infty;H^1)$, and denote
\[
C_w = \| w_{tt} \|^2_{L^{\infty}(0,\infty;H^{-1})}+ \| w_{t} \|_{L^{\infty}(0,\infty;L^2)}^2  \| w \|_{L^{\infty}(0,\infty;H^3)}^2 +  \| \nabla w \|_{L^{\infty}(0,\infty;L^{\infty})}, \ C_q = \| q\|_{L^{\infty}(0,\infty;H^1)}^2.
\]

Let $\{ u^n, \tilde u^n \}$, $n=1,\ 2,\ 3,...$ denote the solution to Algorithm \ref{daproj}, with $\Delta t \le 1$, $\mu \ge \Delta t^{-2}$, $\mu\ge C_w$, and  
$\mu H^2 \le \frac{\nu}{2C_I^2}$.  Then the following bounds hold for any positive $n$:
\begin{align*}
\| \tilde u^n - w^n \| \le  \| u^n - w^n \| &\le C \dt,
\end{align*}
where $C$ depends on problem data and the NSE solution but is independent of $\dt$ and $\mu$.
\end{theorem}

\begin{remark}
Following usual CDA theory and analysis  (see e.g. \cite{LRZ19}), if the initial condition to Algorithm \ref{daproj} were inaccurate, theorem \ref{thm1} would still hold for $n$ large enough.
\end{remark}

\begin{remark}
	If $I_H$ is the projection method onto $X_H$, or we do nudging as in \cite{RZ21}, the analysis can be improved so that is no upper bound on $\mu$, see \cite{GN20}.
\end{remark}

\begin{proof}
Subtracting the NSE at $t=t^{n+1}$ after testing with $v\in X$ from \eqref{pdastep1} and denoting $e^n = w^n - u^n$ and $\tilde e^n = w^n - \tilde u^n$ yields
\begin{align*}
	\frac{1}{\dt}(\enone-\etn,v)+(\utn\cdot\nabla\unone- w^n \cdot\nabla w^{n+1},v) & -(\nabla q^{n+1},v)+\nu(\nabla \enone,\nabla v)+\mu (I_H(\enone),v) \\ &=\dt (w_{tt}(t^\ast) + \dt w_t(t^{**})\cdot\nabla w^{n+1},v), 
\end{align*}
for some $t^*,t^{**} \in (t^n,t^{n+1})$, thanks to Taylor series approximations in the NSE.  Writing the nonlinear terms as
\[
\utn\cdot\nabla\unone- w^n \cdot\nabla w^{n+1} = \tilde e^n \cdot\nabla w^{n+1} + \tilde u^n \cdot\nabla e^{n+1}
\]
and taking $v=e^{n+1}$ vanishes the second nonlinear term and produces
\begin{multline}
\frac{1}{2\Delta t} \left( \| e^{n+1} \|^2 - \| \tilde e^n \|^2 + \| e^{n+1} - \tilde e^n \|^2 \right) + \nu \| \nabla e^{n+1} \|^2 + \mu \|e^{n+1} \|^2 = \\
-(\nabla q,e^{n+1}) - ( \tilde e^n \cdot\nabla w^{n+1},e^{n+1}) + \dt ( w_{tt}(t^\ast) + w_t(t^{**})\cdot\nabla w^{n+1},e^{n+1}) + \mu (e^{n+1}-I_H(e^{n+1}),e^{n+1}), \label{err1}
\end{multline}
after adding and subtracting $e^{n+1}$ to the first argument in the inner product of the nudging term.  We bound the pressure term using Cauchy-Schwarz and Young's inequalities via
\[
-(\nabla q,e^{n+1}) \le \| \nabla q \| \| e^{n+1} \| \le \frac{\mu}{4} \| e^{n+1} \|^2 + \frac{C_q}{\mu}.
\]
For the right hand side nonlinear term, we use H\"older's inequality and regularity of the NSE solution to get
\[
- ( \tilde e^n \cdot\nabla w^{n+1},e^{n+1}) \le \| \tilde e^n \| \| \nabla w^{n+1} \|_{L^{\infty}} \| e^{n+1} \| \le C_w^{1/2} \| \tilde e^n \| \| e^{n+1} \|.
\]
For the third right hand side term in \eqref{err1}, we again use Cauchy-Schwarz and Young's inequalities as well as regularity of the NSE solution to find that
\[
\dt ( w_{tt}(t^\ast) + w_t(t^{**})\cdot\nabla w^{n+1},e^{n+1}) \le \dt \|  w_{tt}(t^\ast) + w_t(t^{**})\cdot\nabla w^{n+1} \|_{-1} \| \nabla e^{n+1} \| \le 
\frac{\nu}{4} \| \nabla e^{n+1} \|^2 + C_w \nu^{-1} \Delta t^2.
\]
To bound the last term in \eqref{err1}, we first apply the Cauchy-Schwarz inequality, then the interpolation estimate \eqref{interp1}, and finally Young's inequality to obtain
\begin{align*}
\mu (e^{n+1}-I_H(e^{n+1}),e^{n+1}) & \le \mu \| e^{n+1}-I_H(e^{n+1}) \| \| e^{n+1} \| \\
& \le \mu C_I H \| \nabla e^{n+1} \| \| e^{n+1} \| \\
& \le \frac{\mu}{4} \| e^{n+1} \|^2 + \mu C_I^2 H^2 \| \nabla e^{n+1} \|^2.
\end{align*}
Collecting the bounds above together with \eqref{err1} provides the estimate
\begin{multline}
\frac{1}{2\Delta t} \left( \| e^{n+1} \|^2 - \| \tilde e^n \|^2 + \| e^{n+1} - \tilde e^n \|^2 \right) + (\nu - \mu C_I^2 H^2) \| \nabla e^{n+1} \|^2 + \frac{\mu}{2} \|e^{n+1} \|^2  \\
\le  
   \mu^{-1}C_q + C_w^{1/2} \| \tilde e^n \| \| e^{n+1} \| +  C_w \nu^{-1} \Delta t^2.
 \label{err2}
\end{multline}
Using the assumption on the parameter $\nu > 2\mu C_I^2 H^2$ and dropping a positive left hand side term reduces the bound to
\begin{equation}
\frac{1}{2\Delta t} \left( \| e^{n+1} \|^2 - \| \tilde e^n \|^2 \right) + \frac{\nu}{2}  \| \nabla e^{n+1} \|^2 + \frac{\mu}{2} \|e^{n+1} \|^2  
\le  
   \mu^{-1} C_q + C_w^{1/2} \| \tilde e^n \| \| e^{n+1} \| +  C_w \nu^{-1} \Delta t^2. \label{err3}
\end{equation}
Next, we subtract $\frac{1}{\Delta t} (w^{n+1},v)$ from both sides of the projection equation to get for $v\in Y$ that
\[
\frac{1}{\dt}(\etnone,v) - (p^{n+1},\nabla \cdot v) = \frac{1}{\dt}(\enone,v),
\]
which implies $L^2$-projection onto the space $Y$, resulting $\| \tilde e^{n+1} \| \le \| e^{n+1} \|$.

Using this in \eqref{err3} provides the bound
 \begin{align*}
\frac{1}{2\Delta t} \left( \| e^{n+1} \|^2 - \| e^n \|^2 \right) + \frac{\nu}{2}  \| \nabla e^{n+1} \|^2 + \frac{\mu}{2} \|e^{n+1} \|^2  
& \le  
   \mu^{-1} C_q + C_w^{1/2} \|  e^n \| \| e^{n+1} \| +  C_w \nu^{-1} \Delta t^2 \\
   & \le (C_w\nu^{-1} + C_q) \dt^2 + \frac{C_w}{2} \| e^n \|^2 + \frac{C_w}{2} \| e^{n+1} \|^2,
\end{align*}
thanks to the assumption $\mu \ge \Delta t^{-2}$.

 Reducing now gives us
 \begin{equation}
 \frac{ 1 + \nu C_p^{-2} \Delta t + \mu\dt - \frac{C_w\dt}{2}  }{ 1+ \frac{C_w\Delta t}{2}  }  \| e^{n+1} \|^2 \le \| e^n \|^2 + \frac{C_w\nu^{-1} + C_q }{1+ \frac{C_w\Delta t}{2}}  \dt^3,
 \end{equation}
 and hence we obtain the bound
 \begin{equation}
 \alpha \| e^{n+1} \|^2 \le \| e^n \|^2 + \frac{C_w\nu^{-1} + C_q}{1+ \frac{C_w\Delta t}{2}}  \dt^3,
 \end{equation}
 where
 \[
 \alpha := \frac{ 1 + \nu C_q^{-2} \Delta t + \frac{\mu}{2} \dt   }{ 1+ \frac{C_w \Delta t}{2}  }  > 1
 \]
 since $\mu>C_w$. Now applying Lemma \ref{series}, and using $e^0=0$ and that $\Delta t\le 1$, after with some simplification we obtain the bound
 \begin{align*}
 \| e^n \|^2 & \le \frac{1}{\alpha-1} \frac{C_w\nu^{-1} +C_q }{1+ \frac{C_w \dt}{2}}  \dt^3 \\
 &\le  \frac{1+ \frac{C_w \Delta t}{2}  }{ \dt (\nu C_p^{-2}  + \frac{\mu}{2} )  }\frac{C_w\nu^{-1} +C_q }{1+ \frac{C_w \dt }{2}}  \dt^3 \\
 &\le  \frac{C_w\nu^{-1} +C_q } { \nu C_p^{-2}+ \frac{\mu}{2}    } \dt^2 \\
 &\le  \frac{C_w\nu^{-1} +C_q } { \nu C_p^{-2}  } \dt^2 \\
 & \le C\dt^2.
 \end{align*}
Taking square roots finishes the proof for $e^n$.  For $\tilde e^n$, this result together with $\| \tilde e^{n} \| \le \| e^{n} \|$ gives the result.
 
\end{proof}

Now that $L^2$ long time first order accuracy of $u$ and $\tilde u$ from Algorithm \ref{daproj} has been established, we can analyze error in other norms as well as the pressure.  The remaining results are for a finite end time $T$, and due to the complicated expressions of constants, $C$ will represent any constant that is independent of $\dt$ and $\mu$.

\begin{theorem}
Under the assumptions of Theorem \ref{thm1} but with finite end time $T$ and number of time steps $M=\frac{T}{\Delta t}$, we have the error bound
\begin{equation}
\left( \dt \sum_{n=1}^M \| \nabla e^{n+1} \|^2 \right)^{1/2} \le C\dt. \label{l2h1err}
\end{equation}
\end{theorem}

\begin{proof}
We begin this proof from \eqref{err2} in the previous theorem's proof:
\begin{multline}
\frac{1}{2\Delta t} \left( \| e^{n+1} \|^2 - \| \tilde e^n \|^2 + \| e^{n+1} - \tilde e^n \|^2 \right) + (\nu - \mu C_I^2 H^2) \| \nabla e^{n+1} \|^2 + \frac{\mu}{2} \|e^{n+1} \|^2  \\
\le  
   \mu^{-1}C_q + C_w^{1/2} \| \tilde e^n \| \| e^{n+1} \| +  C_w \nu^{-1} \Delta t^2.
 \label{err12}
\end{multline}
Using long time first order accuracy of $\tilde u^n$ and $u^n$ along with $\nu>2\mu C_I^2 H^2$ and $\mu\ge \Delta t^{-2}$, we obtain 
\[
\frac{1}{2\Delta t} \left( \| e^{n+1} \|^2 - \| \tilde e^n \|^2 \right) + \frac{\nu}{2}  \| \nabla e^{n+1} \|^2 + \frac{\mu}{2} \|e^{n+1} \|^2  
\le  
   \dt^2 C_q + C_w^{1/2}\dt^2  +  C_w \nu^{-1} \Delta t^2.
\]
Dropping positive left hand side terms and noting $\dt^{-1} \| \tilde e^n \|^2 \le C\dt$ we get the bound
\begin{equation}
 \frac{\nu}{2}  \| \nabla e^{n+1} \|^2 \le C\dt + C\dt^2 \le C\dt.
 \end{equation}
Finally, multiplying both sides by $2\dt$ and summing over time steps produces
\begin{equation}
\nu\dt \sum_{n=1}^M \| \nabla e^{n+1} \|^2 \le C\dt^2, \label{err20}
\end{equation}
which finishes the proof.
\end{proof}

{\color{black}
\section{CDA-Penalty method error analysis}\label{penaltyDAerrorsection}
In this section, we show the long-time accuracy of the CDA-Penalty method. As we do in the previous section, we consider the semi-discrete CDA-Penalty algorithm.

\begin{algorithm}\label{penaltyCDAalg}
	Let $w$ be the solution of \eqref{nse1}-\eqref{nse2} and forcing $f\in L^{\infty}(0,\infty;H^{-1}(\Omega))$.  Then for $u^0=w^0$ and given nudging parameter $\mu \ge 0$, find $ u^{n}$ for n=1,2,3,... via the time stepping algorithm:
\begin{multline}\label{penaltyCDAalgeqn}
\frac{1}{\Delta t}\left(u^{n+1} - u^n,v\right) 
+\tilde b(u^{n},u^{n+1},v)
+ \nu\left(\nabla u^{n+1},\nabla v\right) \\ +\eps^{-1}\left(\nabla\cdot u^{n+1},\nabla\cdot v\right)+\mu\left( I_H(u^{n+1}-w^{n+1}),v\right) = \left(f^{n+1},v\right),\ \forall v\in X.
\end{multline}
\end{algorithm}

\begin{theorem}
Suppose $(w,q)$ is the solution to the NSE with $w\in L^{\infty}(0,\infty;H^3)$, $w_t,w_{tt}\in L^{\infty}(0,\infty;H^1)$, $q\in L^{\infty}(0,\infty;H^1)$, and denote
\[
C_w = \| w_{tt} \|^2_{L^{\infty}(0,\infty;H^{-1})} + \| w_{t} \|_{L^{\infty}(0,\infty;L^2)}^2  \| w \|_{L^{\infty}(0,\infty;H^3)}^2 +  \| \nabla w \|_{L^{\infty}(0,\infty;L^{\infty})}, \ C_q = \| q\|_{L^{\infty}(0,\infty;H^1)}^2.
\]

Let $\{ u^n, \tilde u^n \}$, $n=1,\ 2,\ 3,...$ denote the solution to Algorithm \ref{daproj}, with $\Delta t \le 1$, $\mu \ge \Delta t^{-2}$,
% $\mu\ge C_w$, and  
{\color{black} $ \mu > C \nu^{-1} C_w^2 - C_p^{-2}\frac{\nu}{2}$} and
$\mu H^2 \le \frac{\nu}{2C_I^2}$.  Then the following bounds hold for any positive $n$:
\begin{align*}
 \| u^n - w^n \| &\le C \dt,
\end{align*}
where $C$ depends on problem data and the NSE solution but is independent of $\dt$ and $\mu$.
\end{theorem}

\begin{remark}
From the theorem above, we observe that long time $L^2$ accuracy holds for Algorithm \ref{penaltyCDAalg} and thus so does long time $L^2$ stability.  Thus, under the assumptions of the
theorem, these results immediately can be combined with the Lax-Milgram theorem to establish well-posedness of Algorithm \ref{penaltyCDAalg}.
\end{remark}

\begin{proof}
Subtracting the NSE at $t=t^{n+1}$ from \eqref{penaltyCDAalgeqn}, testing  with $v\in X$ and letting $e^n=w^n-u^n$  provides
\begin{multline*}
\frac{1}{\Delta t}\left(e^{n+1} - e^n,v\right) 
+\tilde b(u^{n},e^{n+1},v)
+\tilde b(e^{n},w^{n+1},v)
+ \nu\left(\nabla e^{n+1},\nabla v\right)
 +\eps^{-1}\left(\nabla\cdot e^{n+1},\nabla\cdot v\right)
+\mu\left( I_H(e^{n+1}),v\right) 
\\
= -(\nabla q^{n+1}, v) + \dt (w_{tt}(t^\ast) + \dt w_t(t^{**})\cdot\nabla w^{n+1},v),
\end{multline*}
for some $t^*,t^{**} \in (t^n,t^{n+1})$, thanks to Taylor series approximations in the NSE.

Setting $v=e^{n+1}$, we derive
\begin{multline*}
\frac{1}{2\Delta t}\left(\|e^{n+1}\|^2 - \|e^n\|^2 +\|e^{n+1} - e^n\|^2 \right) 
+ \tilde b(e^n, w^{n+1},e^{n+1})
+\nu \|\nabla e^{n+1}\|^2 
+\eps^{-1}\|\nabla\cdot e^{n+1}\|^2
 	+\mu\left( I_H(e^{n+1}),v\right) 
 	\\
 	= -(\nabla q^{n+1}, e^{n+1})+\dt (w_{tt}(t^\ast) + \dt w_t(t^{**})\cdot\nabla w^{n+1},v),
\end{multline*}
thanks to the polarization identity. Then, by adding and subtracting $e^{n+1}$ in the first component of the nudging term and dropping positive terms $\|e^{n+1} - e^n\|^2$ and $\eps^{-1}\|\nabla\cdot e^{n+1}\|^2$ , we get that
\begin{multline*}
	\frac{1}{2\Delta t} \|e^{n+1}\|^2 
	+\nu \|\nabla e^{n+1}\|^2 
	+\mu\|e^{n+1}\|^2
	= 
	\frac{1}{2\dt} \|e^n\|^2 \\ 
	- \tilde b(e^n, w^{n+1},e^{n+1})
	-(\nabla q^{n+1}, e^{n+1})
	+\mu (I_H(e^{n+1})-e^{n+1},e^{n+1})
	+\dt (w_{tt}(t^\ast) + \dt w_t(t^{**})\cdot\nabla w^{n+1},v),
\end{multline*}
Using H\"older's and Young's inequality and regularity of the NSE solution, we obtain
\begin{align*}
|-\tilde b(e^n, w^{n+1},e^{n+1})|
&=
\frac{1}{2}
\left(
((e^n\cdot\nabla) w^{n+1},e^{n+1})
-
((e^n\cdot\nabla) e^{n+1},w^{n+1})
\right)
\\
&\leq
\|e^n\| \|\nabla w^{n+1}\|_{L^{3}} \|e^{n+1}\|_{L^6}
+
\|e^n\| \|\nabla e^{n+1}\| \|w^{n+1}\|_{L^{\infty}}
\\
&\leq
\|e^n\| \| w^{n+1}\|_{H^3} \|\nabla e^{n+1}\|
+
\|e^n\| \|\nabla e^{n+1}\| \|w^{n+1}\|_{L^{\infty}}
\\
&\leq
\frac{\nu}{2} \|e^{n+1}\|^2
+
C \nu^{-1 }C_w^2 \|e^n\|^2.
\end{align*}
The rest of the left hand side terms are bounded by following the same analysis as in the proof of Theorem \ref{thm1}.
%Last three terms on the right hand side can be bounded in the same way as in the proof of Theorem \ref{thm1}.
%\[
%-(\nabla q,e^{n+1}) \le \| \nabla q \| \| e^{n+1} \| \le \frac{\mu}{4} \| e^{n+1} \|^2 + \frac{C_q}{\mu}.
%\]
%
%\[
%\dt ( w_{tt}(t^\ast) + w_t(t^{**})\cdot\nabla w^{n+1},e^{n+1})\| \le 
%\frac{\nu}{4} \| \nabla e^{n+1} \|^2 + C_w \nu^{-1} \Delta t^2.
%\]
%
%\begin{align*}
%	\mu (e^{n+1}-I_H(e^{n+1}),e^{n+1}) \le \frac{\mu}{4} \| e^{n+1} \|^2 + \mu C_I^2 H^2 \| \nabla e^{n+1} \|^2.
%\end{align*}
Combining all those bound and  and multiply both sides by $2\dt$ provides

\begin{multline*}
	 \|e^{n+1}\|^2 
	 +\dt\frac{\nu}{2}\|\nabla e^{n+1}\|^2
	+\dt\left(\nu-2\mu C_I^2 H^2 \right)\|\nabla e^{n+1}\|^2 
	+\dt\mu  \|e^{n+1}\|^2
	\\ \leq
	\left(1+ C\dt \nu^{-1} C_w^2\right) \|e^n\|^2 
	+2\dt\mu^{-1} C_q
 +2 C_w \nu^{-1} \Delta t^3.
\end{multline*}
Since $\nu>2\mu C_I^2 H^2$, we drop the positive term $\dt\left(\nu-2\mu C_I^2 H^2 \right)\|\nabla e^{n+1}\|^2 $. Then, by using the Poincar\'e inequality on the left hand side and assuming $\mu\geq\dt^{-2}$, we obtain
\begin{align*}
	\left(1+\dt C_p^{-2}\frac{\nu}{2} +\dt\mu\right)\|e^{n+1}\|^2
	\leq
	\left(1+ C\dt \nu^{-1} C_w^2\right) \|e^n\|^2 
	+\left( 2\mu^{-1} C_q
	+2 C_w \nu^{-1}\right) \Delta t^3.
\end{align*}
Dividing both side by $\left(1+ 2\Delta t C_w\right)$ provides
\begin{align*}
	\alpha\|e^{n+1}\|^2
	\leq
	\|e^n\|^2 
	+\frac{ 2\mu^{-1} C_q
	+2 C_w \nu^{-1}}{1+ C\dt \nu^{-1} C_w^2}\Delta t^3,
\end{align*}
where 
\begin{align*}
\alpha\coloneqq \frac{1+\dt C_p^{-2}\frac{\nu}{2} +\dt\mu}{1+ C\dt \nu^{-1} C_w^2 }>1,
\end{align*}
since $ \mu > C \nu^{-1} C_w^2 - C_p^{-2}\frac{\nu}{2}$.

Finally, by Lemma \ref{series}, we obtain
\begin{align*}
	\|e^{n+1}\|^2
	&\leq
\frac{1}{\alpha-1}
\frac{ 2\mu^{-1} C_q
		+2 C_w \nu^{-1}}{1+ C\dt \nu^{-1} C_w^2 }\Delta t^3
	\\
%		&\leq
%	\frac{1+ C\dt \nu^{-1} C_w^2 }{\dt\left( C_p^{-2}\frac{\nu}{2} +\mu-C \nu^{-1} C_w^2 \right)}
%	\frac{ 2\mu^{-1} C_q
%		+2 C_w \nu^{-1}}{1+ C\dt \nu^{-1} C_w^2 }\Delta t^3
		&\leq
\frac{ 2\mu^{-1} C_q
	+2 C_w \nu^{-1}}{\dt\left( C_p^{-2}\frac{\nu}{2} +\mu-C \nu^{-1} C_w^2 \right)}\Delta t^3
	\\&\leq C\dt^2
	,
\end{align*}
since $\dt<1$.
Taking the square of both sides finishes the proof.

\end{proof}

\section{Numerical Results}\label{experimentsection}
In this section, we illustrate the above theory with two numerical tests, an analytical test with known true solution and channel flow past a block (a.k.a. square cylinder \cite{ST96}).  For these tests we compute with both projection and penalty methods.  In all results, we observe that CDA improves solution accuracy, and with enough measurement data it effectively removes the splitting error of the projection method and the consistency error of the penalty method.  

For the spatial discretization, we use a regular, conforming triangulation of the domain $\Omega$ which is denoted by $\tau_h$. Let $X_h \subset X$ and $Q_h\subset Q$ be an inf-sup stable pair of discrete velocity-pressure spaces. We take $X_h = X\cap P_2(\tau_h)$ and $Q_h = Q\cap P_{1}(\tau_h)$ Taylor-Hood or Scott-Vogelius elements in our tests, however our results in the previous sections are extendable to most other inf-sup stable element choices.

For all our penalty method tests, we use $\eps=1$.  While this is a very large penalty and smaller penalty values would lead to less consistency error and thus better accuracy, smaller penalty values also lead to linear systems that are very difficult to solve for large scale problems since the matrix arising from the grad-div term is singular.  With $\eps=1$, however, solving these systems can still be reasonably efficient \cite{HR13,benzi,OT14}.  

\subsection{Known analytical solution}

For our first experiment, we illustrate the accuracy theory above for Algorithms \ref{daproj} and \ref{penaltyCDAalg} to a chosen analytical solution
\begin{align*}
	u(x,y,t)&= \left( e^t \cos(y),e^t \sin(x) \right) \\
	p(x,y,t)&=(x-y)(1+t),
\end{align*}
on the unit square domain. We consider $(P_2, P_1)$ Taylor-Hood elements for velocity and pressure.  The initial velocity is taken as $u(0)=0$ in the CDA tests. The source term $f$ is calculated from the chosen solution and the NSE.
\begin{figure}[H]
	\centering
	\includegraphics[width = .48\textwidth, height=.35\textwidth]{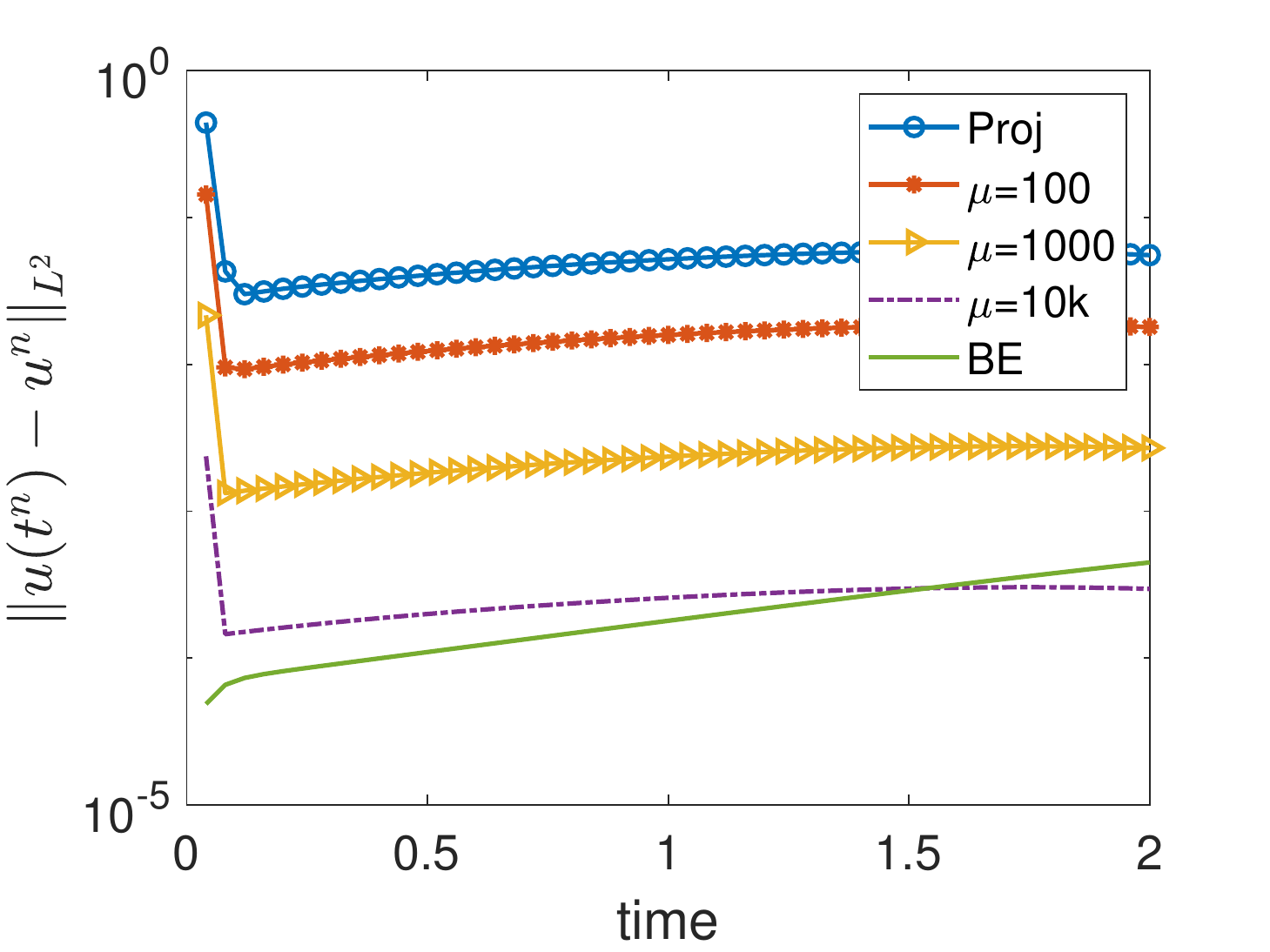}
	\caption{Shown above is the $L^2$ difference to the true solution versus time, for finite element solution of NSE with backward Euler time discretization and Algorithm \ref{daproj} with varying $\mu$ for zero initial velocities.}\label{L2BDerror}
\end{figure}

We first test Algorithm \ref{daproj} (CDA Projection method) on [0,2] with $\Delta t=0.05$ on an $h=1/128$ uniform triangular mesh, with $H=1/32$ grid for the measurement data. Figure \ref{L2BDerror} shows $L^2(\Omega)$ error versus time for varying $\mu$, and for comparison also with the usual backward Euler (BE) FEM using the nodal interpolant of $u(0)$ as the initial condition.  We observe that as $\mu$ increases, the error approaches that of BE (which is known to be first order in $\Delta t$), with it reaching the same level of accuracy when $\mu=10^5$ despite having a very inaccurate initial condition.

We repeat this test for Algorithm \ref{penaltyCDAalg} (CDA Penalty method), and results are shown in figure \ref{penaltyCDAalg} as $L^2$ error versus time.  Results are similar to that of CDA Projection, with improvement in accuracy as $\mu$ increases and finally achieving the same accuracy as BE once $\mu=10^5$.

\begin{figure}[H]
	\centering
		\includegraphics[width = .48\textwidth, height=.35\textwidth]{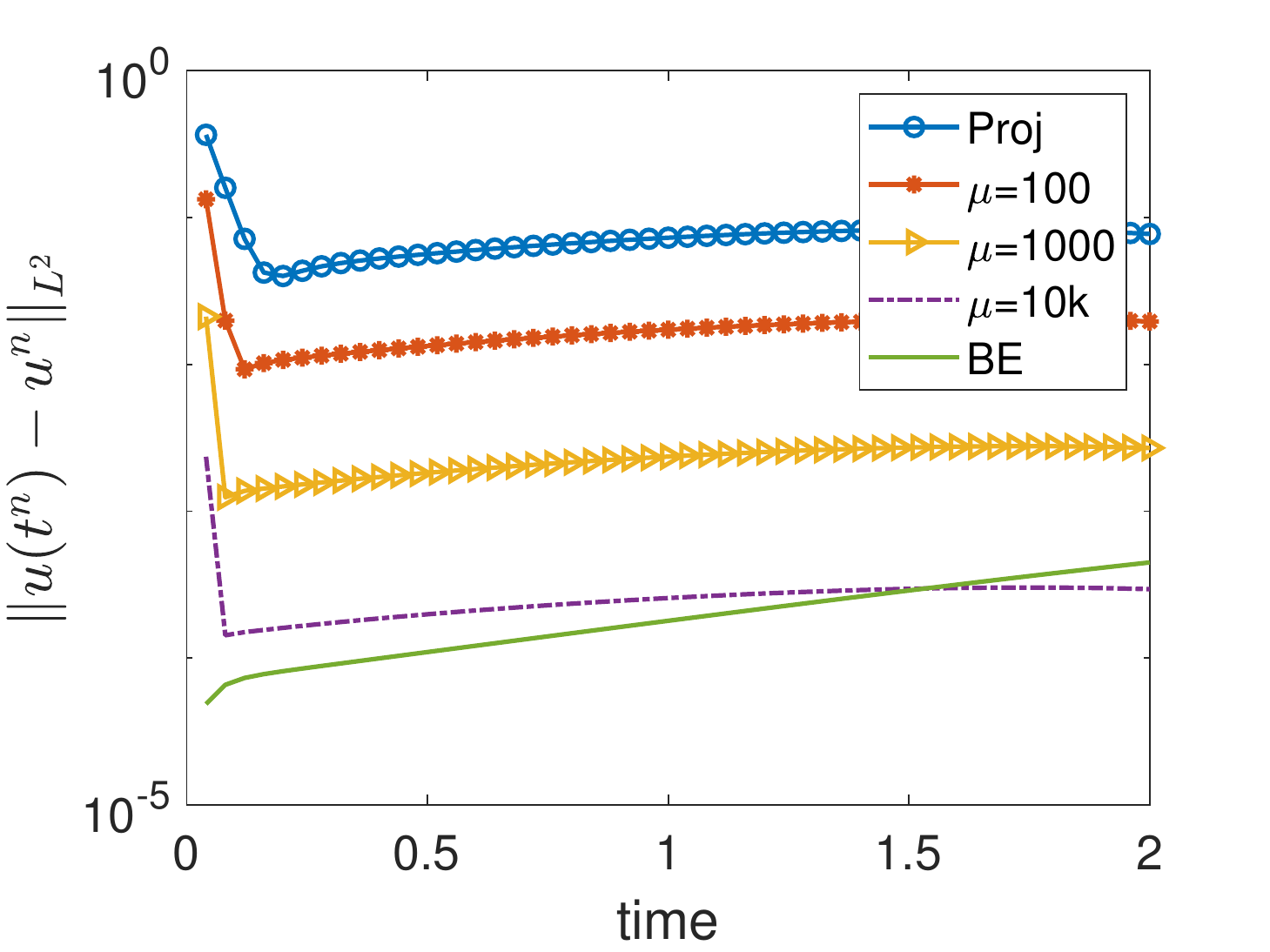}
	\caption{Shown above is the $L^2$ difference to the true solution versus time, for finite element solution of NSE with backward Euler time discretization and Algorithm  \ref{penaltyCDAalg}  with $\eps=1$ and varying $\mu$ for zero initial velocities.}\label{L2errorBEPenalty}
\end{figure}

\subsection{Channel flow past a block}

The second experiment tests the proposed data assimilation methods on the problem of channel flow past a block. Many experimental and numerical studies can be found in the literature \cite{TGO15, R97, SDN99}. The domain of the problem consists of a $2.2\times0.41$ rectangular channel, and a block having a side length of $0.1$ centered at $(0.2, 0.2)$ from the bottom left corner of the rectangle. See  Figure \ref{fig:blockdomain} for a diagram of the domain.

No-slip velocity boundary and homogeneous normal boundary conditions are enforced on the block and walls for step 1  and step 2 of the projection scheme with CDA respectively. The inflow and outflow flow profiles are given by
\begin{align}
	u_1(0,y,t)=&u_1(2.2,y,t)=\frac{6}{0.41^2}y(0.41-y),\nonumber\\
	u_2(0,y,t)=&u_2(2.2,y,t)=0.
\end{align}
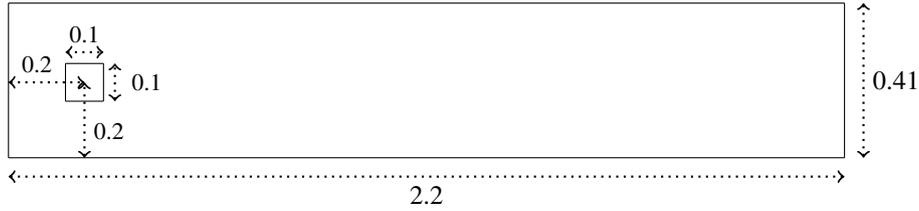
\begin{figure}[H]
	\centering
	\begin{tikzpicture}[scale=5]
		\draw (0,0) -- (2.2,0);
		\draw [dotted, <->, thick] (0,-0.05) -- (2.2,-0.05);
		\draw (1.1,-0.05) node[below]{$2.2$};
		\draw (2.2,0) -- (2.2,0.41);
		\draw [dotted, <->, thick] (2.25,0) -- (2.25,0.41);
		\draw (2.25,0.205) node[right]{$0.41$};
		\draw (0,0.41) -- (2.2,0.41);
	%	\draw (1.9,0.30) node{outlet};
		\draw (0,0) -- (0,0.41);
		\draw[dotted, <->, thick] (0.2,0)--(0.2,0.2);
		\draw (0.15,0.15) -- (0.25,0.15);
		\draw (0.25,0.15) -- (0.25,0.25);
		\draw (0.15,0.25) -- (0.25,0.25);
		\draw (0.15,0.15) -- (0.15,0.25);
		\draw [dotted, <->, thick] (0,0.2)--(0.2,0.2);
	%	\draw (-0.2,0.3)node{inlet};
		\draw [dotted, <->, thick]  (0.28,0.15)--(0.28,0.25);
		\draw [dotted, <->, thick]  (0.15,0.28)--(0.25,0.28);
		\draw (0.3,0.2)node[right]{\small$0.1$};
		\draw (0.2,0.28)node[above]{\small$0.1$};
		\draw (0.075,0.20)node[above]{\small$0.2$};
		\draw (0.2,0.07)node[right]{\small$0.2$};
	\end{tikzpicture}
	\caption{The domain for the channel flow past a cylinder numerical experiment}\label{fig:blockdomain}
\end{figure}
The kinematic viscosity is taken to be $\nu=10^{-3}$ and external force $f=0$. Quantities of interest for this flow are lift and drag coefficients. We use $(P_2, P_1^{disc})$ Scott-Vogelius elements on a barycenter refined Delaunay mesh that provides $19.4k$ velocity and $14.3k$ pressure degrees of freedom.  We take $T=10$ as the end time in our computations.  The BDF2-FEM scheme with $\Delta t=0.002$ sufficiently resolves the solution on this mesh, and we use this as the resolved solution from which to draw measurements from and make comparisons to.  Below, we give lift $c_l(t)$ and drag $c_d(t)$ calculations from tests with and without CDA:

\begin{align*}
c_d(t)&=\frac{2}{\rho L U^2_{max}} \int_{S} \left(\rho \nu \frac{\partial u_{ts}}{\partial n} n_y - p(t) n_x\right)\ dS,\\
c_l(t)&=-\frac{2}{\rho L U^2_{max}} \int_{S} \left(\rho \nu \frac{\partial u_{ts}}{\partial n} n_x - p(t) n_y\right)\ dS,\\
\end{align*}
where $U_{max}$ is the maximum mean flow, $L$ is the diameter of the cylinder, $n =
(nx; ny)^T$ is the normal vector on surface $S$ and $u_{ts}$ is the tangential velocity, \cite{ST96}.

In addition to testing the first order methods in Algorithms  \ref{daproj} and \ref{penaltyCDAalg}, we also test their BDF2 analogues, which are given in PDE form below. \\

CDA Proj Step 1 with BDF2: Find $u^{n+1}$:
\begin{align*}
	\frac{1}{2\Delta t}\left( 3 u^{n+1} -4\tilde u^{n} +\tilde u^{kn-1} \right) + \tilde u^{n}\cdot\nabla u^{n+1} - \nu\Delta u^{n+1}  + -\nabla p^{n} + \mu  I_H(u^{n+1} - w^{n+1}) & = f^{n+1},\\
	u^{n+1}|_{\partial\Omega}&=0.
\end{align*}

CDA Proj Step 2 with BDF2: Project $u^{n+1}$ into the divergence-free space
\begin{align*}
	\frac{1}{2\Delta t} \left(3\tilde u^{n+1} - 3u^{n+1}\right)+ \nabla\left( p^{n+1}-p^n\right)  &=0, \\
	\nabla \cdot \tilde u^{n+1} &=0,\\
	\tilde u^{n+1} \cdot n|_{\partial\Omega}&=0.
\end{align*}

CDA-Penalty with BDF2: Find $u^{n+1}$:
\begin{equation}
	\begin{aligned}
		\frac{1}{2\Delta t}\left( 3 u^{n+1} -4 u^{n} + u^{k-1} \right) + \tilde B(u^n, u^{n+1})
		- \nu\Delta u^{n+1} -\eps^{-1}\nabla\div u^{n+1}  
		+\mu I_H(u^{n+1}-w^{n+1}) & = f^{n+1},\nonumber\\
		u^{n+1}|_{\partial\Omega}&=0.\nonumber
	\end{aligned}
\end{equation}

\subsubsection{Projection method results}

In this subsection, we test both BE and BDF2 projection methods.  First, we test with no CDA and varying time step sizes.  Results are shown in figure \ref{endraglift_BE_proj}, as drag and lift coefficients versus time.  For BE projection, results are bad for each choice of $\Delta t$: although there is improvement as $\Delta t$ decreases, even with $\Delta t=10^{-4}$, results are quite inaccurate.  With BDF2 projection, results are significantly better, and with $\Delta t=10^{-4}$ the results match that of the resolved solution.
\begin{figure}[H]
	\centering
BE Projection (no CDA)\\
	\includegraphics[width = .48\textwidth, height=.35\textwidth,viewport=0 0 500 430, clip]{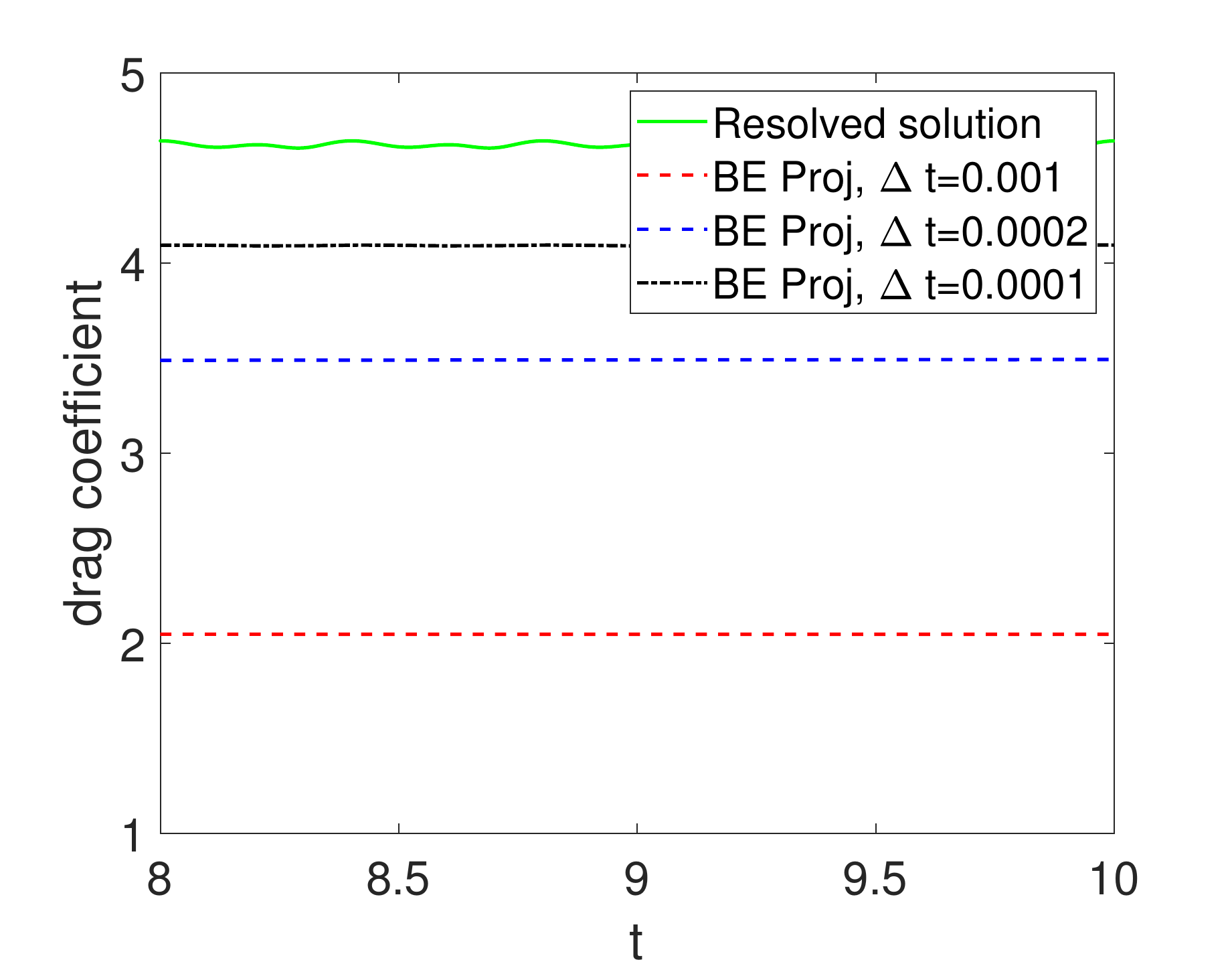}
	\includegraphics[width = .48\textwidth, height=.35\textwidth,viewport=0 0 500 430, clip]{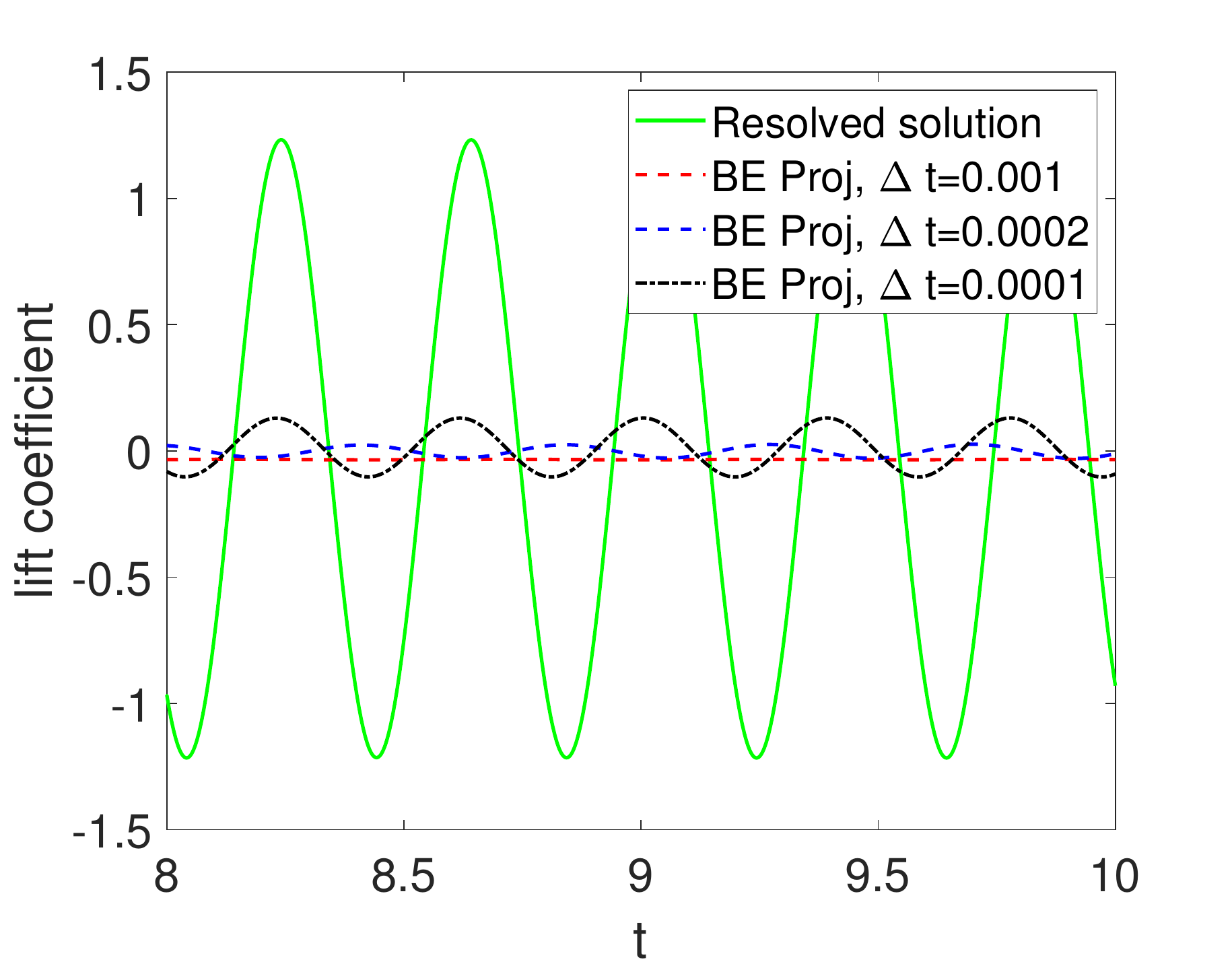}\\
BDF2 Projection (no CDA)\\
	\includegraphics[width = .48\textwidth, height=.35\textwidth,viewport=0 0 500 430, clip]{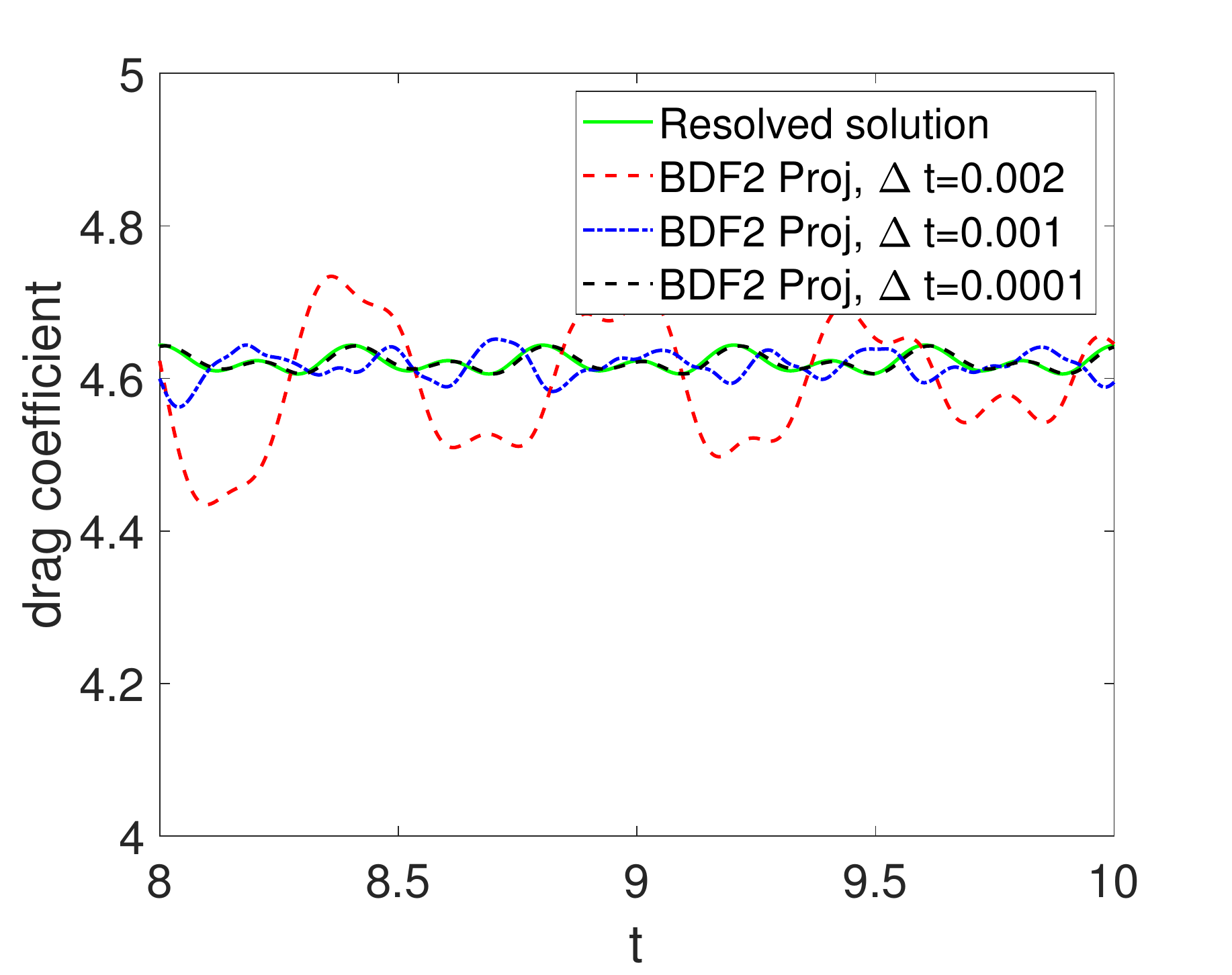}
	\includegraphics[width = .48\textwidth, height=.35\textwidth,viewport=0 0 500 430, clip]{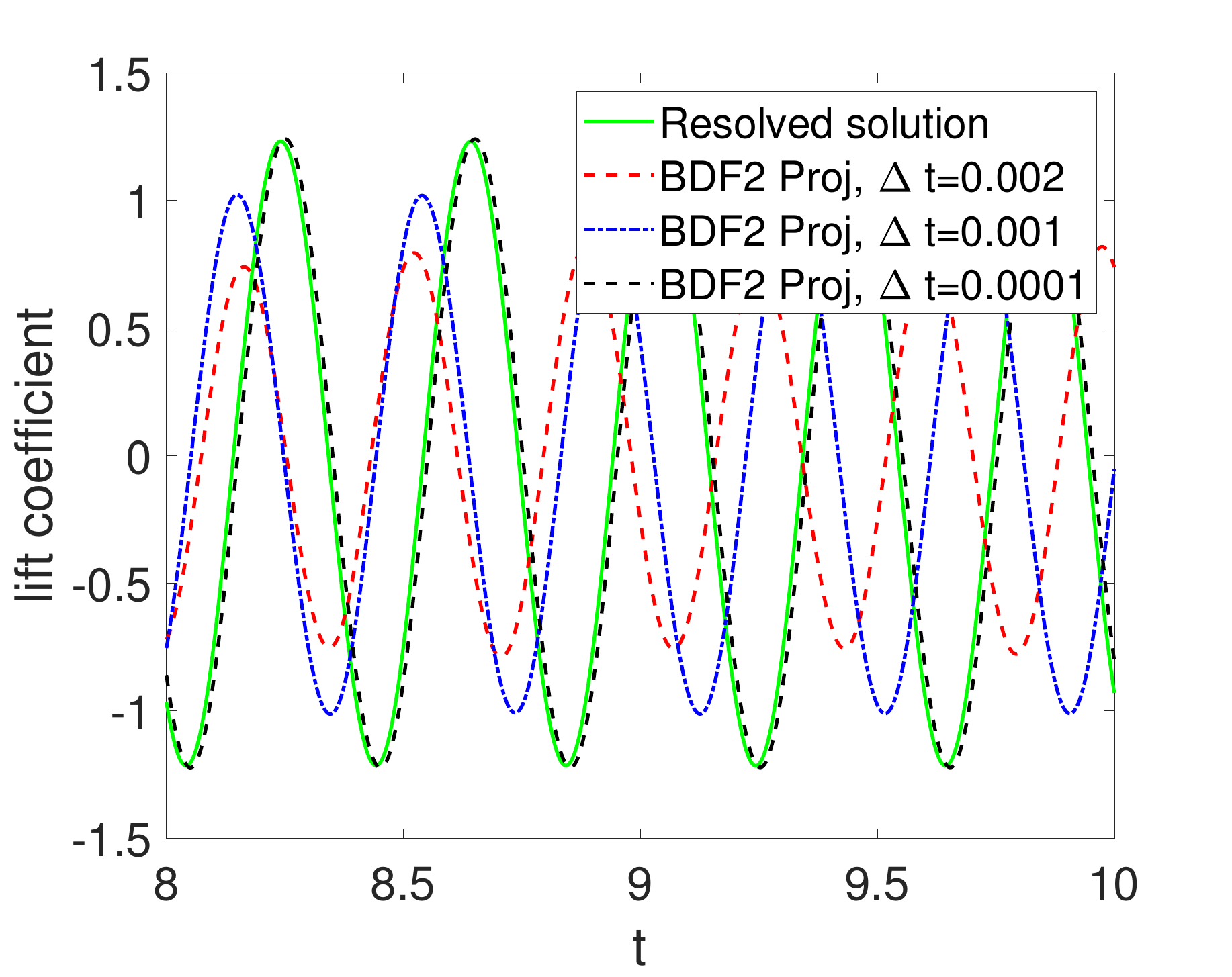}

	\caption{Shown above are the drag (left) and lift(right) coefficients versus time for backward Euler projection method (top) and BDF2 projection method (bottom) with varying time step $\dt$ and no CDA.}\label{endraglift_BE_proj}
\end{figure}

Next we consider the CDA projection methods, with nudging parameter $\mu=1000$ and time step size $\Delta t=0.002$ (far larger than what is needed to match the resolved solution when no CDA is used), with varying number of data measurement points $N^2$. In figure \ref{endraglift_BD_projDA}, we observe that BE Projection is nearly as accurate as the resolved solution only when $N=61$.  BDF2 Projection, on the other hand, is accurate even when $N=21$.  In all cases, CDA provides significant improvement in accuracy, and with BDF2 can provide results as good as the resolved solution with a more reasonable number of measurement points than BE projection requires.

\begin{figure}[H]
	\centering
	BE Projection CDA \\
	\includegraphics[width = .48\textwidth, height=.35\textwidth,viewport=0 0 500 430, clip]{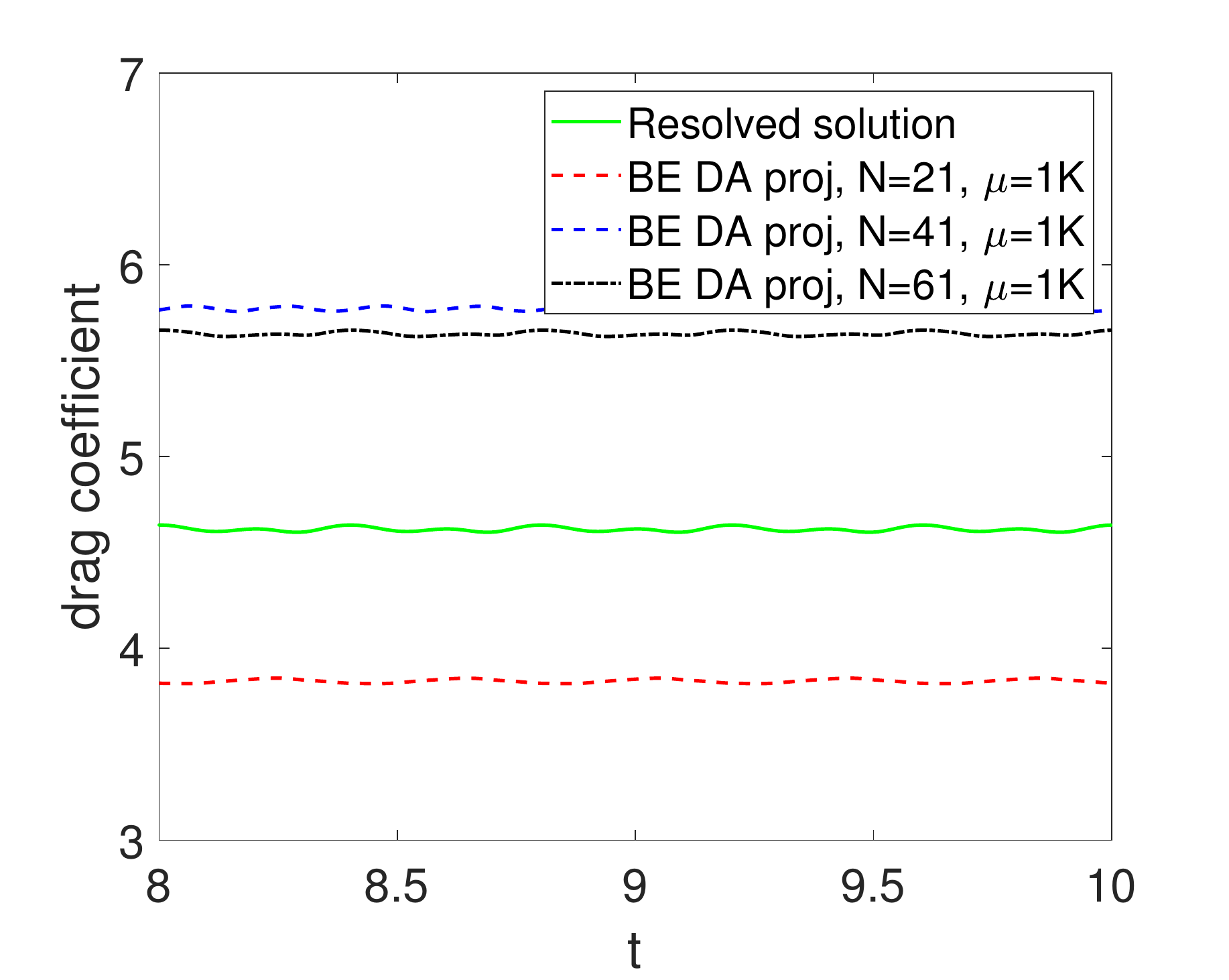}
	\includegraphics[width = .48\textwidth, height=.35\textwidth,viewport=0 0 500 430, clip]{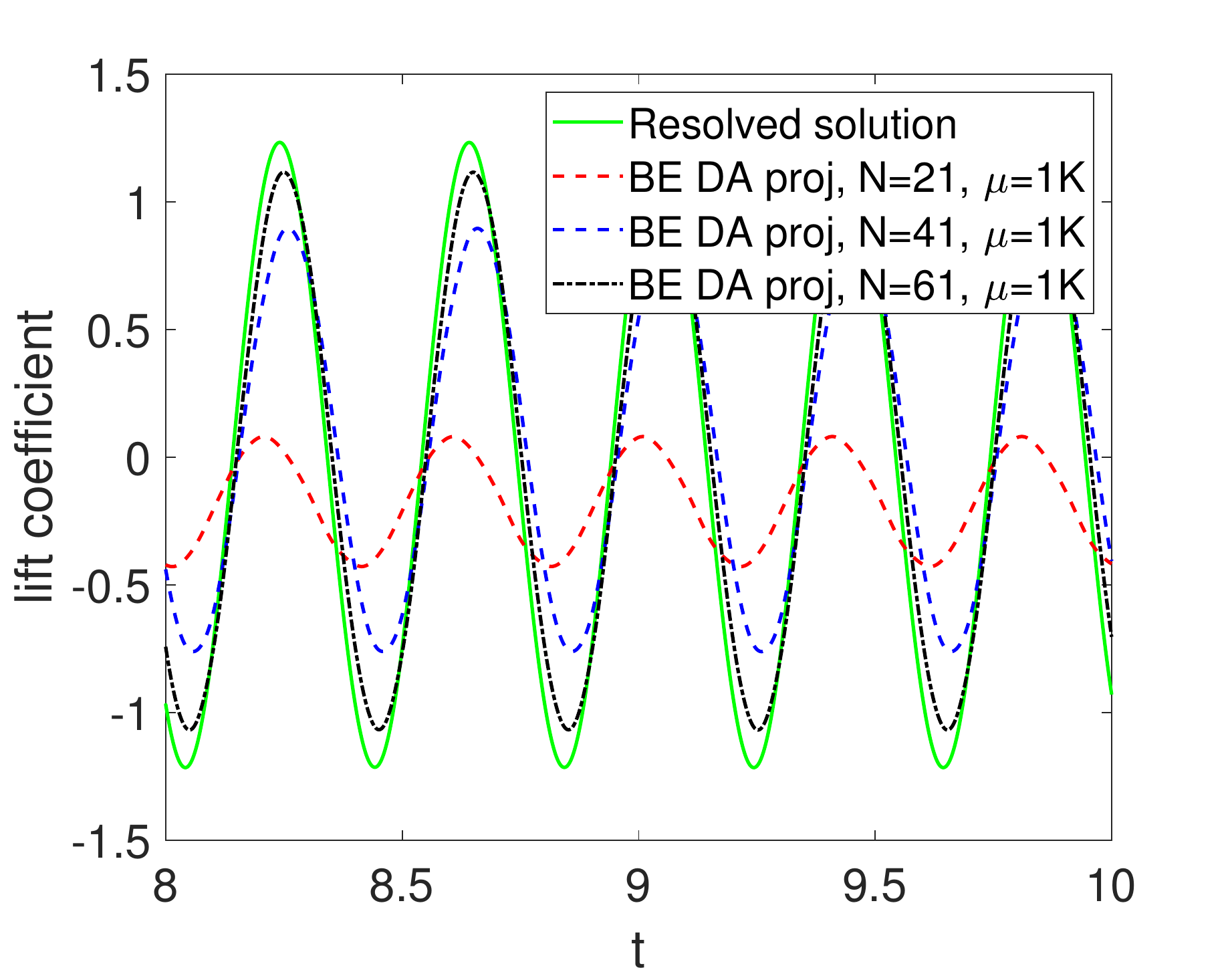}\\
	BDF2 Projection CDA \\
	\includegraphics[width = .48\textwidth, height=.35\textwidth,viewport=0 0 500 430, clip]{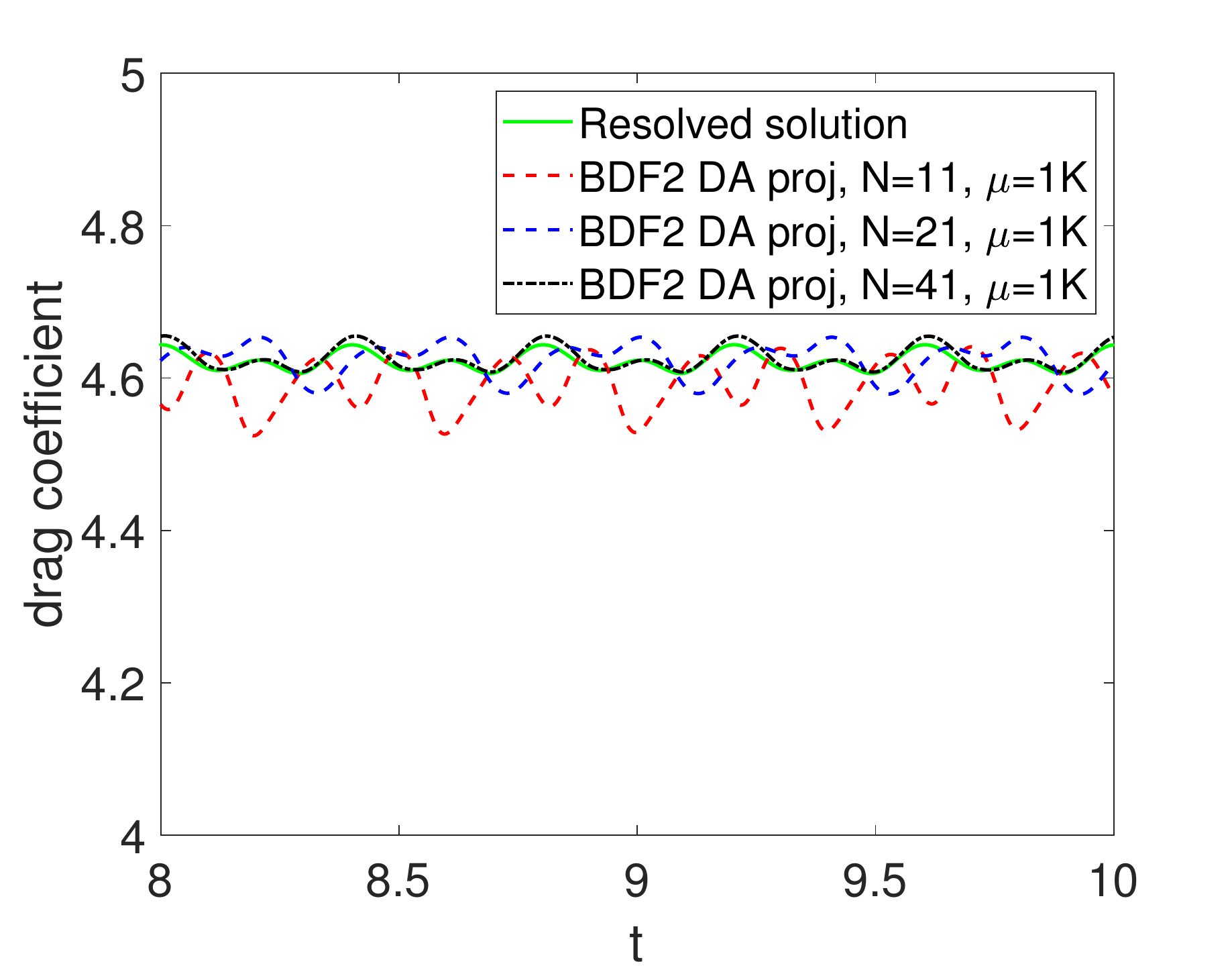}
	\includegraphics[width = .48\textwidth, height=.35\textwidth,viewport=0 0 500 430, clip]{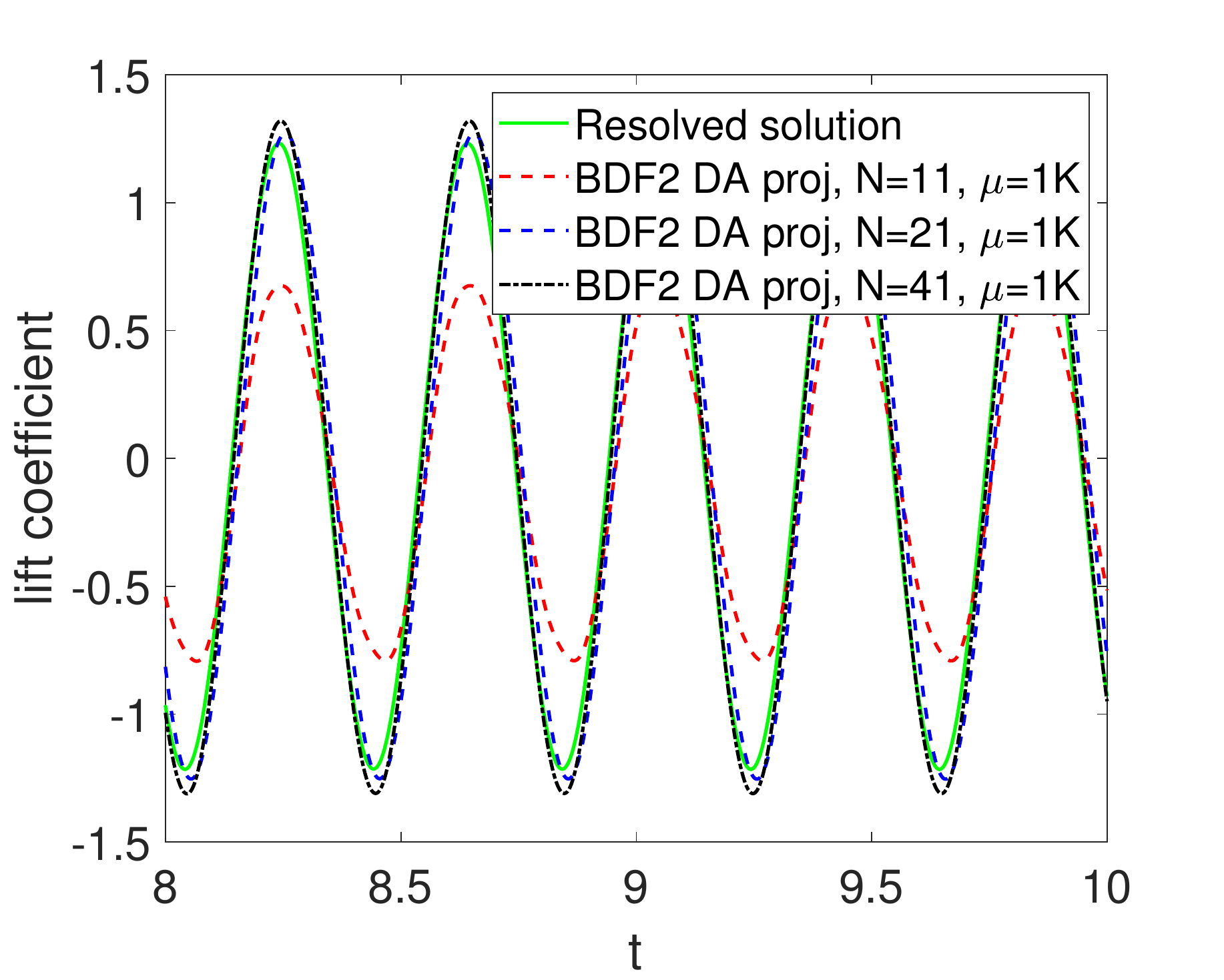}	
	
	\caption{Shown above are the drag(left) and lift(right) coefficients versus time for CDA Projection methods with varying $N$ and $\mu=1000$.}\label{endraglift_BD_projDA}
\end{figure}

\subsubsection{Penalty method with data assimilation using Backward Euler time stepping}
 
We now repeat the tests done with CDA Projection for CDA penalty, now using $\mu=10$ (larger $\mu$ did not improve results).  Results for lift and drag are shown in
figure \ref{endraglift_BD_penDA}, and we observe similar results as for CDA Projection: for BE Penalty, $N=61$ is required to achieve accuracy near that of the resolved solution and for BDF2 Penalty $N=41$ is needed.

\begin{figure}[H]
	\centering
	Backward Euler Penalty CDA \\
	\includegraphics[width = .48\textwidth, height=.35\textwidth,viewport=0 0 500 390, clip]{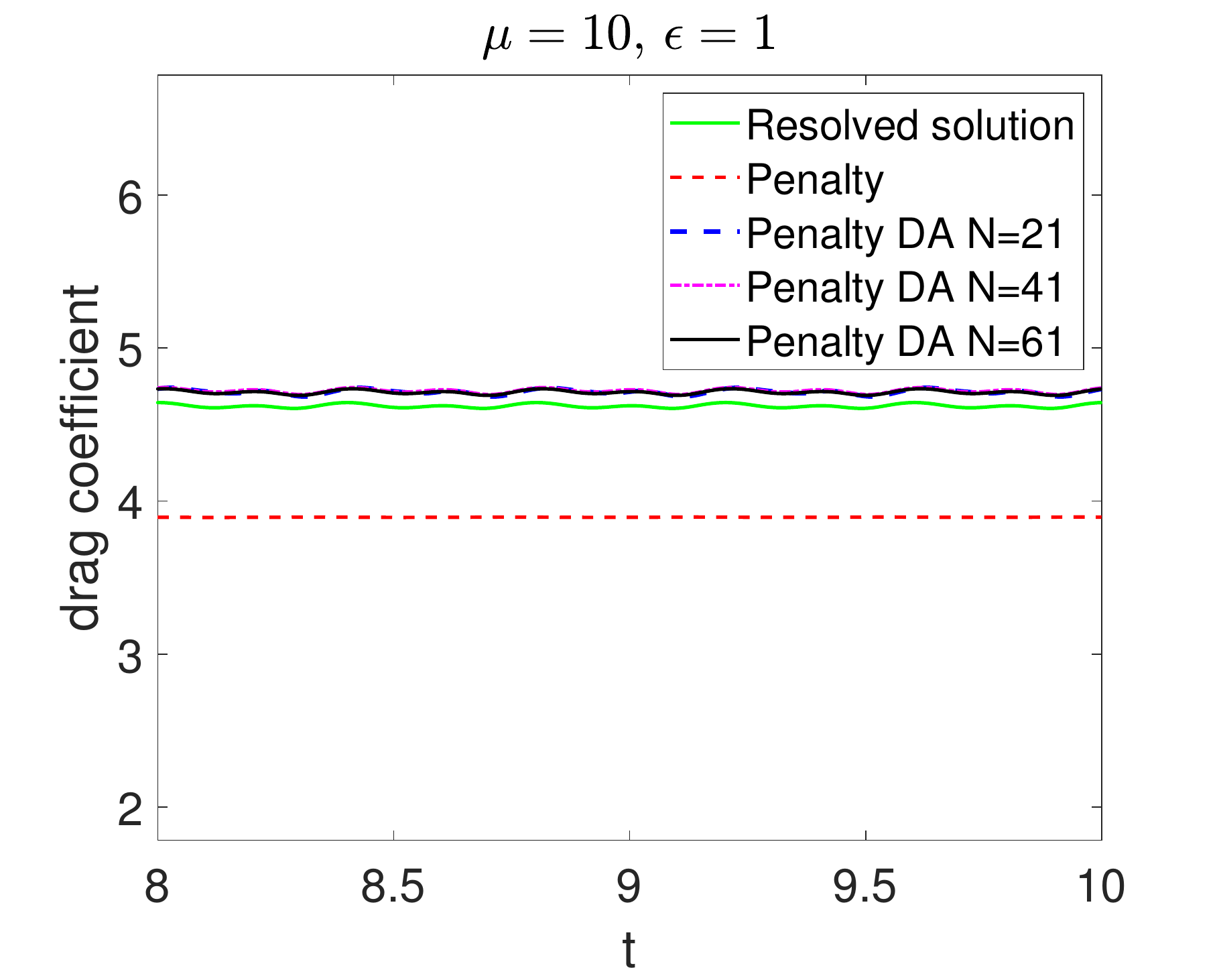}
	\includegraphics[width = .48\textwidth, height=.35\textwidth,viewport=0 0 500 390, clip]{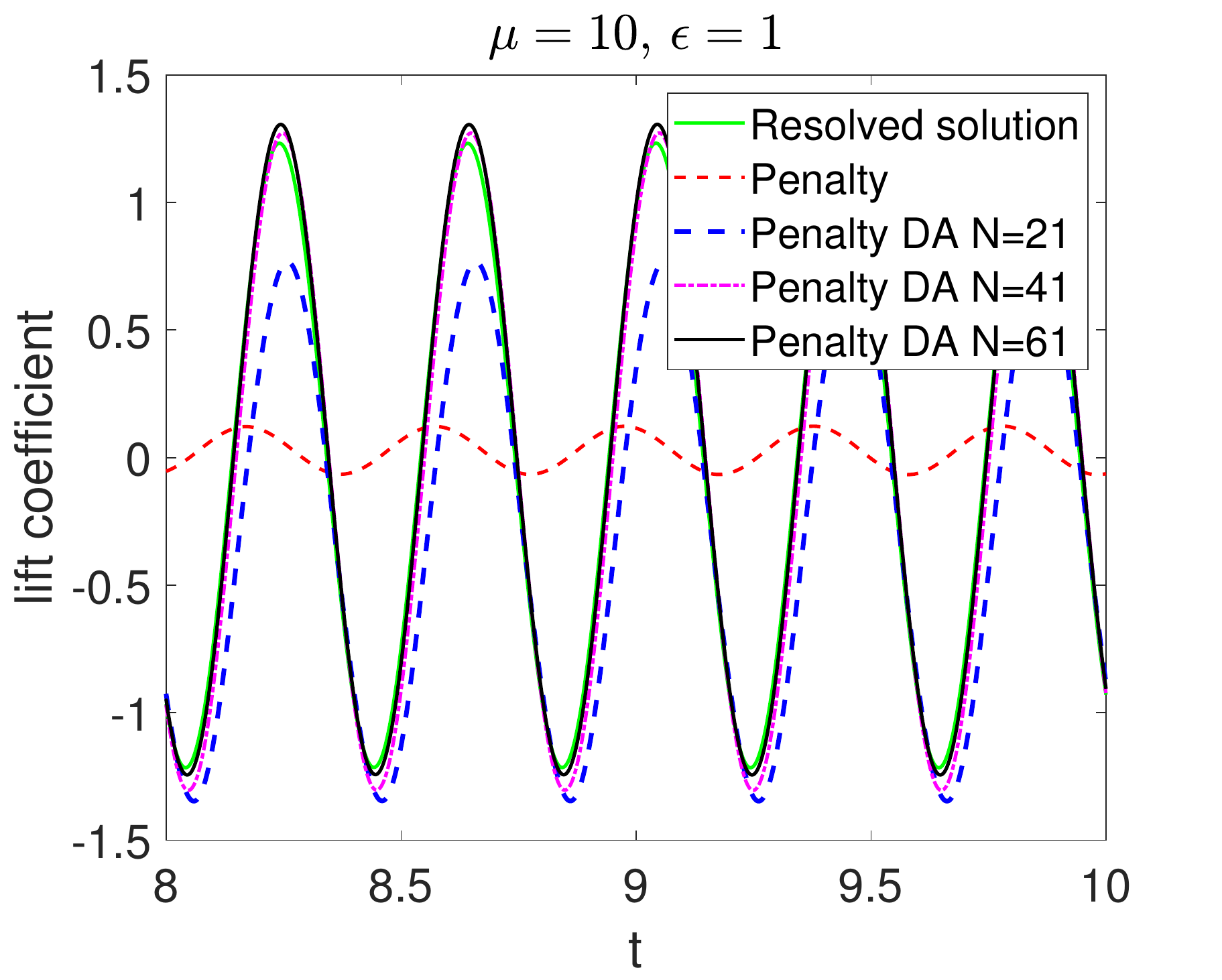}
	BDF2 Penalty CDA\\
	\includegraphics[width = .48\textwidth, height=.35\textwidth,viewport=0 0 500 390, clip]{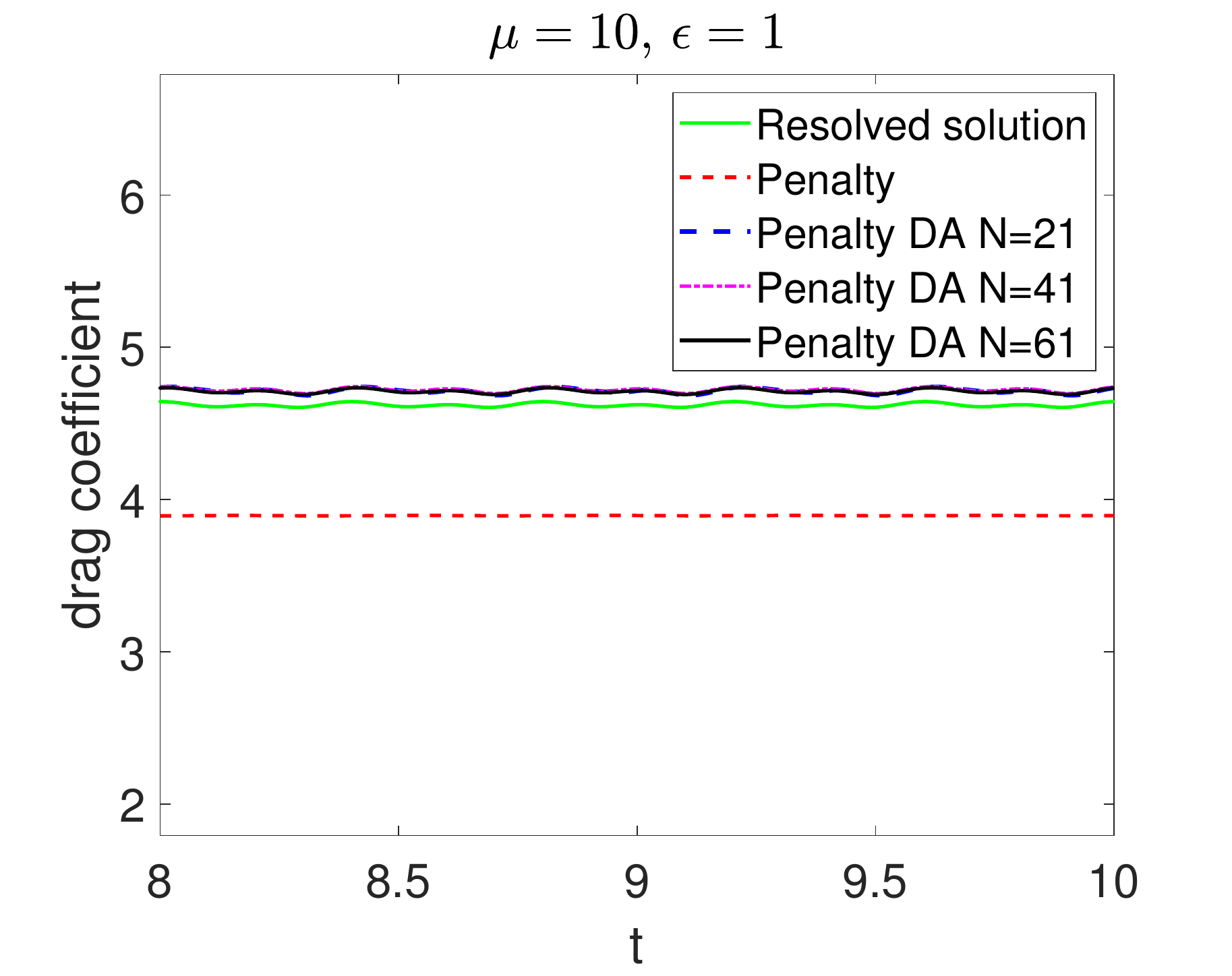}
	\includegraphics[width = .48\textwidth, height=.35\textwidth,viewport=0 0 500 390, clip]{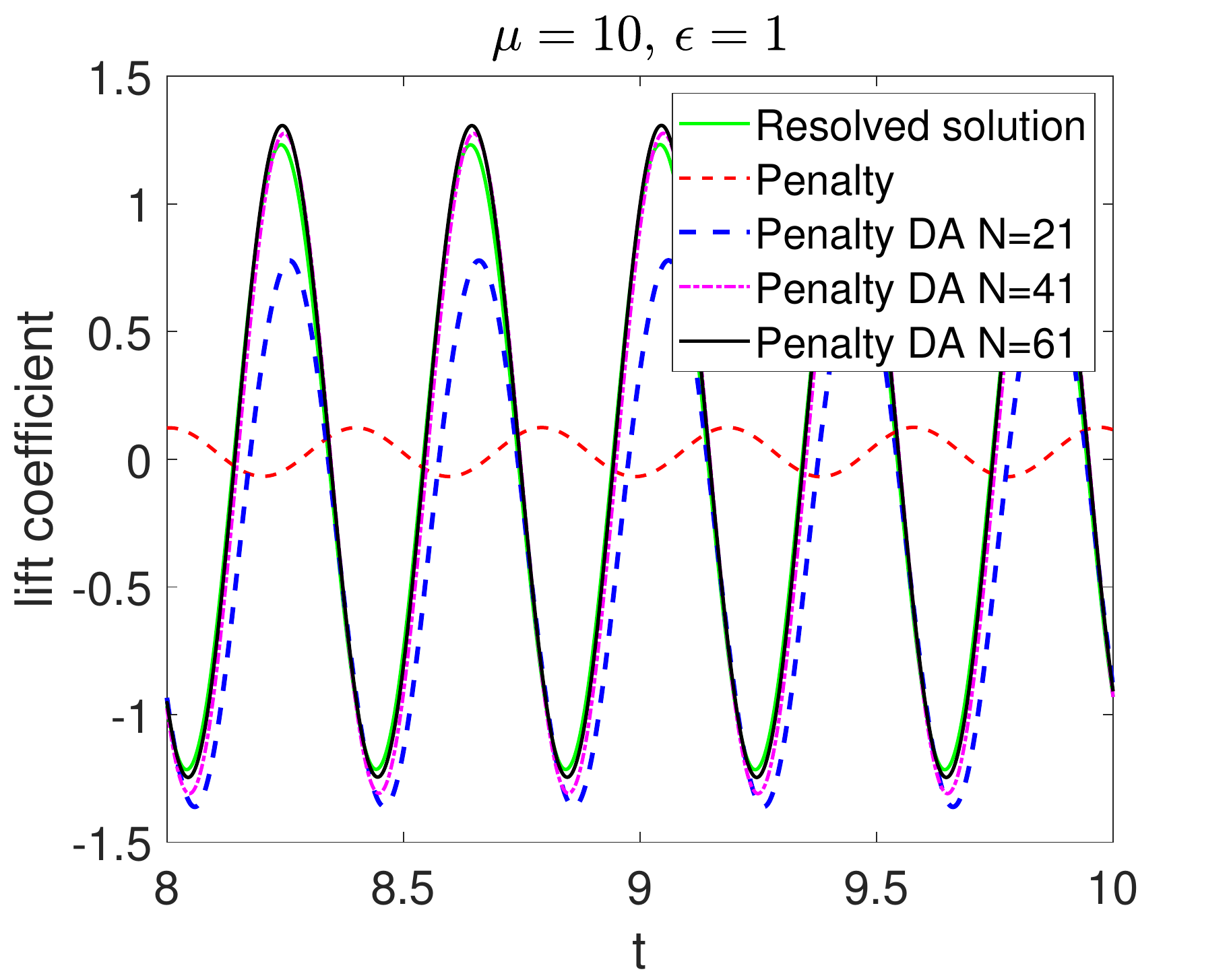}
	\caption{Shown above are the drag(left) and lift(right) coefficients versus time for CDA Penalty methods with varying $N$ and $\mu=10$ and $\eps=1$}\label{endraglift_BD_penDA}
\end{figure}

\section{Conclusions}
We studied herein continuous data assimilation (CDA) applied to the projection and penalty methods for the Navier-Stokes equations. We proved that CDA enables long time optimally accurate solutions by removing the splitting error arising in projection methods and the modeling error in penalty methods.  Numerical tests illustrated the theory well, and major improvements in accuracy from CDA were observed.  These tests also showed that CDA can allow for larger time step sizes and larger penalty parameters without harming accuracy.  

For future work, one may consider CDA applied to other types of splitting or approximation methods such as Yosida or ACT algebraic splitting methods for NSE \cite{RX17,SV06,HH02,ViguerieThesis}, to determine if their splitting errors can be reduced as well.

\section{Acknowledgements}
All authors were partially supported by NSF grant DMS 2152623.

\bibliographystyle{plain}
\bibliography{graddiv}
%\bibliography{paperbib2}

\end{document}